\documentclass[10pt, reqno]{amsart}
\usepackage{amssymb}
\allowdisplaybreaks

\newcommand{\R}{{\mathbb{R}}}
\newcommand{\C}{{\mathbb{C}}}
\newcommand{\Z}{{\mathbb{Z}}}

\let\Re=\undefined\DeclareMathOperator*{\Re}{Re}
\let\Im=\undefined\DeclareMathOperator*{\Im}{Im}
\DeclareMathOperator*{\cosec}{cosec}

\DeclareMathOperator*{\Id}{Id}
\DeclareMathOperator{\E}{\mathbb{E}}
\DeclareMathOperator{\Prob}{{{\mathbb{P}}}}

\newcommand{\HS}{H_\text{\rm LHM}}
\newcommand{\HH}{H_\textrm{Heis}}
\newcommand{\SO}{\textrm{SO}}

\newcommand{\qtq}[1]{\quad\text{#1}\quad}
\newcommand{\eps}{\varepsilon}

\newcommand{\Haar}{\textrm{Haar}}

\newtheorem{theorem}{Theorem}[section]
\newtheorem{proposition}[theorem]{Proposition}
\newtheorem{lemma}[theorem]{Lemma}

\theoremstyle{definition}
\newtheorem{definition}[theorem]{Definition}
\newtheorem{remark}[theorem]{Remark}

\newcounter{smalllist}

\newenvironment{SLH}{\begin{list}{{\rm(\roman{smalllist})\hss}}{%
\setlength{\topsep}{0mm}\setlength{\parsep}{0mm}\setlength{\itemsep}{0mm}%
\setlength{\labelwidth}{2.0em}\setlength{\itemindent}{0em}\setlength{\leftmargin}{2.5em}\usecounter{smalllist}%
}}{\end{list}}
\def\tempsepline{{}-\!\!\!-\!\!\!-\!\!\!-\!\!\!-\!\!\!-\!\!\!-}
\def\tempsep{\smallskip\centerline{$\tempsepline\!\!\!\!\hbox to 0mm{\hss//\hss}\!\tempsepline$}\smallskip}

\begin{document}

\title[Invariant measures for integrable spin chains and discrete NLS]{Invariant measures for integrable spin chains and integrable discrete NLS}

\author[Y. Angelopoulos]{Yannis Angelopoulos}
\address{Department of Mathematics, UCLA, Los Angeles, USA}
\email{yannis@math.ucla.edu}

\author[R. Killip]{Rowan Killip}
\address{Department of Mathematics, UCLA, Los Angeles, USA}
\email{killip@math.ucla.edu}

\author[M. Visan]{Monica Visan}
\address{Department of Mathematics, UCLA, Los Angeles, USA}
\email{visan@math.ucla.edu}

\begin{abstract}
We consider discrete analogues of two well-known open problems regarding invariant measures for dispersive PDE, namely, the invariance of the Gibbs measure for the continuum (classical) Heisenberg model and the invariance of white noise under focusing cubic NLS.  These continuum models are completely integrable and connected by the Hasimoto transform; correspondingly, we focus our attention on discretizations that are also completely integrable and also connected by a discrete Hasimoto transform.   We consider these models on the infinite lattice $\Z$.

Concretely, for a completely integrable variant of the classical Heisenberg spin chain model (introduced independently by Haldane, Ishimori, and Sklyanin) we prove the existence and uniqueness of solutions for initial data following a Gibbs law (which we show is unique) and show that the Gibbs measure is preserved under these dynamics.  In the setting of the focusing Ablowitz--Ladik system, we prove invariance of a measure that we will show is the appropriate discrete analogue of white noise.

We also include a thorough discussion of the Poisson geometry associated to the discrete Hasimoto transform introduced by Ishimori that connects the two models studied in this article.
\end{abstract}

\maketitle

%%%%%%%%%%%%%%%%%%%%%%%%%%%%%%%%%%%%%%%%%%%%%%%%%%%%%%%%%
\section{Introduction}
%%%%%%%%%%%%%%%%%%%%%%%%%%%%%%%%%%%%%%%%%%%%%%%%%%%%%%%%%

The research detailed in this paper began with the consideration of the following problem: Can one prove invariance of the Gibbs measure for the one-dimensional continuum (classical) Heisenberg model:
\begin{equation}\label{SME}
\partial_t \vec S = - \vec S \times \Delta \vec S
\end{equation}
where $\vec S:\R_t\times\R_x \to \mathbb S^2$ describes the configuration of spins, $\times$ denotes the cross-product, and $\Delta=\partial_x^2$ is the spatial Laplacian.

This model is a special case of the Schr\"odinger maps equation (where general K\"ahler targets are allowed).  It is also associated with the names of Landau--Lifshitz (see \cite{LL} or \cite[\S69]{LL9}), who also introduced a damping term into these dynamics, and of Gilbert (see \cite{Gilbert}), who further refined their theory at high damping.  It is natural to also include an external magnetic field in \eqref{SME}; however, this would only complicate a problem that we already do not know how to solve.

Gibbs measure provides a statistical description of a physical system at thermal equilibrium and is dictated by the inverse temperature $\beta>0$, the Hamiltonian (or energy functional), and the underlying symplectic volume.

From a physical point of view, \eqref{SME} arises as the continuum limit of the classical Heisenberg spin-chain model
\begin{equation}\label{Heis}
\frac{d\ }{dt} \vec S_n = - \vec S_n \times\bigl( \vec S_{n+1} +  \vec S_{n-1} \bigr),
\end{equation}
describing the dynamics of a chain of spins $\vec S:\R_t\times\Z\to \mathbb S^2$.  This dynamics is Hamiltonian, being induced by the energy functional
\begin{equation}\label{HeisH}
\HH := \sum_{n\in \Z}  \tfrac12|\vec S_n - \vec S_{n+1}|^2
\end{equation}
with respect to the Poisson structure \eqref{PB} below, which is merely the vestige (in classical mechanics) of the standard (quantum mechanical) commutation relations for spins.  It is shown in \cite{Semiclassical} that the quantum mechanical spin chain reduces to this classical model in the limit of large spin per site. 

\begin{definition}[Poisson bracket]\label{D:pb} For fields $\vec S:\Z \to\mathbb{S}^2\subset\R^3$, we define the Poisson bracket via
\begin{equation}\label{PB}
  \bigl\{ \vec a \cdot \vec S_n,\, \vec b\cdot \vec S_m \bigr\} = \delta_{nm} \, \vec a \cdot (\vec S_n \times \vec b ).
\end{equation}
\end{definition}

The symplectic form associated to this Poisson bracket is the sum of the standard surface area on each copy of $\mathbb S^2$.  As it comes from a (closed) symplectic structure, this Poisson bracket is immediately guaranteed to obey the Jacobi identity, although this can also be checked directly via Lagrange's identity for the cross product.

Analogously, the continuum model \eqref{SME} is naturally associated to the Hamiltonian
$$
\int_{\R} |\nabla \vec S(x)|^2\,dx,
$$
which (formally at least) tells us that the associated Gibbs measure simply corresponds to Brownian paths on the sphere.  The key difficulty associated with the problem posed in the first paragraph of this paper is not to make sense of the Gibbs measure, but rather, to be able to make sense of the dynamics \eqref{SME} for such irregular data.

The study of Hamiltonian PDE at low regularity has been a topic of intensive study for many years now and has made it possible to prove the existence of dynamics for initial data sampled from Gibbs measures and thence the invariance (under the flow) of these Gibbs measures for a variety of Hamiltonian PDE.  We note, in particular, the pioneering work (on both fronts) of Bourgain, surveyed in \cite{Bourg:book}.

At this moment, the most powerful method for studying the Schr\"odinger maps equation at low regularity is via the Hasimoto transform.  Discovered in the study of vortex tubes in \cite{Hasimoto} and first applied to \eqref{SME} in \cite{Laks}, this mapping transforms solutions to \eqref{SME} into solutions to the focusing cubic NLS:
\begin{equation}\label{NLS}
i \psi_t  = - \partial_x^2 \psi - \tfrac12 |\psi|^2 \psi .
\end{equation}
Concretely, viewing $x\mapsto \vec S(t,x)$ as the field of tangents to an arc-length parameterized curve in $\R^3$, one defines
\begin{equation}\label{Hasi}
\psi(t,x) = \kappa(t,x) \exp\biggl\{-i\int_{-\infty}^x \tau(t,x')\,dx'\biggr\}
\end{equation}
where $\kappa$ denotes the curvature of the curve and $\tau$ its torsion.  Note that the energy of the spin wave is carried over to the mass of the solution to NLS,
\begin{equation}\label{HisM}
\int_\R |\nabla \vec S(x)|^2\,dx = \int_\R |\psi(x)|^2 \,dx,
\end{equation}
rather than to the traditional Hamiltonian for \eqref{NLS}.  Evidently, the Hasimoto map is not a Poisson map with respect to the \emph{standard} Poisson structure associated to NLS.

The presence of a second (compatible) Poisson structure for \eqref{NLS} is indicative of the well-known complete integrability of NLS (cf. \cite{Magri}).  The equation \eqref{SME} has also been shown to be completely integrable, both directly \cite{Takh} and via Hasimoto-type transformations \cite{Laks,ZakhTakh}.  While the problem of constructing dynamics for \eqref{SME} with initial data sampled from the Gibbs measure seems out of reach at the current moment, the complete integrability of this equation is, at least, propitious.

The original calculations used in deriving the Hasimoto transformation involve use of the Frenet--Serret formulae for curves.  As is well-known, this approach to the differential geometry of curves is poorly adapted to vanishing curvature.  These difficulties can be averted by adopting a parallel frame (cf. \cite{Bishop}) along the curve.  Indeed, this approach has lead to the development of Hasimoto-like transformations in the context of general K\"ahler targets, as well as for higher dimensional arrays of spins; see \cite{CSU,Ding,Koiso,RodRub}.

Regarded as a mapping of individual states (rather than trajectories), it is not difficult to see that the Hasimoto transform maps Brownian paths on the sphere to white noise on the line.  Setting aside whether this can be extended to trajectories (in any sense), this raises the question of studying NLS with white noise initial data.  This problem is well-known and currently open, for focusing and defocusing nonlinearities, both on the line and on the circle.  In fact, one would formally expect white noise measure to be invariant under the NLS flow.  For the state of the art in the low-regularity problem for NLS, we refer the reader to \cite{CarlesKappeler,Christ,CollianderOh,GrunHerr,GuoOh,KVZ,Kishimoto,KochTataru}, as well as \cite{BanVega,JerrardSmets} which study low-regularity problems originating directly from \eqref{SME}.  We include here several references considering problems on the circle or, what is equivalent, for periodic initial data.  As white noise constitutes non-decaying (indeed ergodic) data on the line, there is a strong analogy with the circle case.

One thing that is clearly understood in the circle setting is that one must renormalize \eqref{NLS} to have any hope of treating data at regularities below $L^2$; see \cite{GuoOh}.  At the very least, one must employ Wick ordering, which amounts to removing an infinite phase shift from solutions to the equation.

Once one accepts that renormalization may be necessary to make sense of the model \eqref{SME} for Gibbs distributed initial data, then one is compelled to return to the basic physics.  Not only should one endeavor to renormalize in a physical way, but the break-down of the effective model should also be regarded as casting doubt on its derivation from more elementary principles.  Concretely, one is lead to ask if \eqref{SME} is the proper continuum limit of \eqref{Heis} in the setting of thermal equilibrium.

For smooth initial data, the convergence of \eqref{Heis} to \eqref{SME} is shown rigourously in \cite{Bardos&c}.  Our hesitation in assuming that this result extends to low regularity data is most easily explained through consideration of the continuum limit of the discrete linear Schr\"odinger equation
\begin{equation}\label{DLS}
i\partial_t \psi_n = - \bigl( \psi_{n+1} + \psi_{n-1} \bigr),
\end{equation}
with initial data constructed by choosing each $\psi_n$ independently and identically distributed according to a complex Gaussian law.  It is easily shown (by Fourier transformation) that this measure is invariant under the flow.
Now, this measure and indeed these dynamics are left invariant by the transformation
$$
\psi_n \mapsto (-1)^n \bar\psi_n
$$
which shows that low-frequencies (slowly varying sequences) and very high frequencies (slowly varying modulus with alternating signs) contribute equally to the problem in question.  However, it is only for the low frequencies that one would traditionally conflate the Laplacian with its finite difference approximation.  For the model \eqref{DLS} with white-noise initial data, one is lead to posit that the continuum limit should be described (at the very least) by a \emph{pair} of linear Schr\"odinger equations: one for the low frequencies and one for the high frequencies.

While it is fair to say that the process of inverting the Hasimoto transform is one of integration, which would suppress the high frequencies, our immediate discussion has centered around the linear model \eqref{DLS}.  Nonlinearities would couple the low- and high-frequency portions of the solution and thus we cannot discount the possibility that the high-frequency components impact the low-frequency dynamics in a non-trivial way.

We should caution the reader that the preceding discussion is heuristic and that we are not asserting the existence of a Hasimoto-like transform attendant to \eqref{Heis}.  Nonetheless, we shall soon discuss a discrete spin chain model and a discrete nonlinear Schr\"odinger equation that are connected by such a Hasimoto-like transformation; moreover, both are completely integrable.  On the other hand, numerical evidence \cite{RobThom} suggests that the model \eqref{Heis} is not completely integrable.

Low regularity problems in dispersive PDE are inherently difficult, notwithstanding the additional difficulties stemming from passing to the continuum limit of a discrete model.  Past experience suggests the greatest chance of success if one works with a completely integrable model, which led us to seek out discrete analogues of \eqref{SME} and \eqref{NLS} that retain complete integrability and which are connected by a Hasimoto-like transformation.  This pursuit does not represent a disparagement of \eqref{Heis}, but rather, the belief that it may be more fruitfully treated as a perturbation of such a completely integrable analogue, rather than attacked directly.

Our search for an integrable discrete analogue of \eqref{Heis} was a very short one.  It is clearly documented in \cite{FT:book}:
\begin{equation}\label{E:S}
\frac{d\ }{dt} \vec S_n = - \vec S_n \times\biggl(\frac{2\vec S_{n+1}}{1+\vec S_n\cdot\vec S_{n+1}} + \frac{2\vec S_{n-1}}{1+\vec S_n\cdot\vec S_{n-1}} \biggr),
\end{equation}
which has Hamiltonian
\begin{equation}\label{HS}
\HS := \sum_n  -2\log\bigl(1 - \tfrac14|\vec S_n - \vec S_{n+1}|^2 \bigr)
\end{equation}
with respect to the standard Poisson structure \eqref{PB}.  Following this reference, we will refer to this model as the Lattice Heisenberg Model (LHM), which appeared independently in three papers \cite{Haldane,Ishimori,Skly} in the same year.

The book \cite{FT:book} also describes (following \cite{IK}) a transformation of the LHM to a completely integrable form of discrete NLS.  However, this mapping is essentially a stereographic projection at each position along the lattice and so is unlike the Hasimoto transform, which acts like a derivative.  It is not difficult to obtain a discrete analogue of the Hasimoto transformation, starting from \eqref{E:S} and mimicking the arguments in \cite{Hasimoto}; see the next section.  However, the answer (found by a different method) appears already in \cite{Ishimori}, which shows that the LHM can be transformed to the (focusing) Ablowitz--Ladik system,
\begin{equation}\label{E:AL}
i\tfrac{d\ }{dt} \alpha_n =-\bigl(1+ |\alpha_n|^2 \bigr)  \bigl[\alpha_{n+1}+\alpha_{n-1}\bigr] + 2\alpha_n.
\end{equation}
This model was introduced in \cite{AL} as an integrable discretization of \eqref{NLS}.

Informed by the preceding discussion, our immediate goals with regard to the models \eqref{E:S} and \eqref{E:AL} are now clear:
\begin{SLH}
\item Construct (unique) Gibbs measures for \eqref{E:S}.

\item Prove the existence and uniqueness of the dynamics \eqref{E:S} with initial data sampled from this measure.

\item Show that these dynamics leave said Gibbs measures invariant.

\item Determine a suitable discrete analogue of white-noise that is connected to the Gibbs measure for \eqref{E:S} via a discrete Hasimoto transformation.

\item Show that \eqref{E:AL} is well-posed for initial data sampled from this `white noise' measure and that the dynamics \eqref{E:AL} leaves this measure invariant.
\end{SLH}
This is what will be achieved in this paper.  The rather more challenging problem of taking a continuum limit in these results remains our ambition for the future. We note that the approach to constructing invariant measures for NLS by taking a continuum limit of the Ablowitz--Ladik system has already been shown to be successful in \cite{Vaninsky}.  In that paper, Vaninsky considers the defocusing problem on the circle and constructs an invariant measure associated to the conservation law at one degree of regularity higher than the Hamiltonian.  For convergence in the deterministic setting, see \cite{HongYang}, which works in the energy space, and references therein.

The existence and uniqueness of Gibbs measures for \eqref{E:S} will be proved in Proposition~\ref{P:Gibbs}.  While the prevailing method for proving dynamical invariance of Gibbs measures is based on finite-dimensional approximation, we eschew this methodology for the construction of the measure.  Instead, we adopt the \emph{intrinsic} definition of Gibbs measures introduced by Dobrushin, Lanford, and Rulle; see \cite{Dob,LR}.   We prove uniqueness of such Gibbs measures by using the Perron--Frobenius Theorem to show that the underlying Markov chain is mixing; see \eqref{1137}.

In order to prove invariance of the Gibbs measure, we need a more direct construction than the abstract existence and uniqueness given by Proposition~\ref{P:Gibbs}.  This is effected by using the discrete Hasimoto transformation in reverse to construct initial data for \eqref{E:S} from initial data for \eqref{E:AL}.  In fact, we will also construct solutions to \eqref{E:S} by this method, namely, by first constructing solutions to \eqref{E:AL} and then transferring them to \eqref{E:S}.  The virtues of employing the discrete Hasimoto transform here are the same as in the continuum case --- it transforms a quasilinear problem into a semilinear one, which makes it much easier to control both individual solutions and differences between pairs of solutions.

Up to now, we have avoided addressing one of the main deficiencies of the Hasimoto transform, namely, its failure to admit an invariant definition, both in the sense of dynamically invariant and in the sense of being independent of arbitrary choices.  This problem stems from the incompatible gauge invariances of the two equations involved: The spin models (both continuum and discrete) have a global $SO(3)$ gauge invariance corresponding to a collective rigid rotation of all the spins, while \eqref{NLS} and \eqref{E:AL} have global $U(1)\cong SO(2)$ phase invariance.  In the study of individual solutions, this nuisance is usually handled by fixing a gauge for the initial data and propagating the resulting frame through time, as necessary.  For \emph{statistical ensembles of solutions} (as considered here) this is unsatisfactory --- it leads to measurability issues and non-invariant measures (due to dynamical modifications of the gauge).  The remedy we adopt here is to \emph{randomize the gauge} and show that this randomization is dynamically invariant.

Our discussion of the discrete Hasimoto transform is divided into two parts: In Section~\ref{S:DH} we present its construction by paralleling the classical approach of \cite{Hasimoto}.  This will allow us to elucidate the Poisson structure of the discrete Hasimoto transformation more fully than appears to have been done before.  On the other hand, this approach breaks down whenever consecutive spins are parallel --- this is the discrete analogue of the problem of vanishing curvature in the Frenet--Serret description of curves.

In Section~\ref{S:AL to S}, we revisit the discrete Hasimoto transform in a manner parallel to modern treatments of the continuum version, which are based on parallel frames.  This approach does not suffer from problems with vanishing curvature; moreover, it is well-suited to randomization of the gauge.  Neither this approach nor that presented in Section~\ref{S:DH} is very close to that adopted in \cite{Ishimori}, where the discrete Hasimoto transform was first discovered.

Already in Section~\ref{S:DH}, it is possible to deduce what distribution should be assigned to initial data for the Ablowitz--Ladik system so that it corresponds to the Gibbs measure for \eqref{E:S} via the discrete Hasimoto transform.  The answer is given in \eqref{E:D:wn}.  The values at each site are statistically independent, as one might well imagine for a measure mimicking white noise.  However, their distribution is \emph{not} Gaussian  --- it has very long tails.  In fact, at inverse temperature $\beta>0$, we have $\alpha_n\in L^p(d\;\!\mathbb P)$ if and only if $p<2+\beta$.

In Section~\ref{S:AL GWP}, we first prove almost sure existence and uniqueness of solutions to \eqref{E:AL} for initial data sampled from the measure \eqref{E:D:wn}.  This is Theorem~\ref{T:AL GWP}.  We then show that this flow preserves the measure \eqref{E:D:wn}; this is Theorem~\ref{T:wn invariance}.  The key idea is to take a limit (uniform on bounded sets in space-time) of solutions to spatial truncations of the equation.  For such finite systems, global well-posedness follows from standard ODE techniques; see Proposition~\ref{P:trunc}.  Note that these methods cannot be applied in infinite volume.  First, as RHS\eqref{E:AL} is not globally Lipschitz, one can only hope to apply contraction mapping on a small time interval whose length is dictated by the size of the data.  But as our initial data is ergodic under translation, every possible local configuration will occur with positive density somewhere; thus no time interval is short enough to apply contraction mapping if one works globally in space.  Secondly, to pass from local to global well-posedness, one would like to apply conservation laws; however, all conserved quantities are infinite in this case.

The method we employ is a close analogue of that used by Bourgain \cite{Bourg:IMNLS} to construct solutions to defocusing NLS on the line with initial data sampled from Gibbs measure.  The principal novelty in this paper is in the implementation, where subtleties arise from the long tails in the distribution of the initial data.

The climax of the paper is Section~\ref{S:5} where we prove existence and uniqueness of the Gibbs measure for \eqref{E:S}, construct unique solutions associated to such initial data, and prove the resulting dynamics leaves the Gibbs measure invariant.  In summary, we prove

\begin{theorem}[Invariance of the Gibbs measure for LHM]\label{T:S GWP}
Fix $\beta>0$. For almost every initial data distributed according to the Gibbs measure $d\mu_{Gibbs}^\beta$, there exists a unique global good solution to the spin chain model \eqref{E:S}.  Moreover, the Gibbs measure $d\mu_{Gibbs}^\beta$ is left invariant by the flow of \eqref{E:S}.
\end{theorem}

%%%%%%%%%%%%%%%%%%%%%%%%%%%%%%%%%%%%%%%%%%%%%%%%%%%%%%%%%
\section{The discrete Hasimoto transform}\label{S:DH}
%%%%%%%%%%%%%%%%%%%%%%%%%%%%%%%%%%%%%%%%%%%%%%%%%%%%%%%%%

Our goal in this section is to develop the discrete Hasimoto transform following closely the methodology expounded in the original work of Hasimoto \cite{Hasimoto}.

\begin{definition}\label{D:thetagamma} For a field $\vec S:\Z\to \mathbb{S}^2$, with no two consecutive spins parallel or antiparallel, we define coordinates $\theta_n\in(0,\pi)$ and $\gamma_n\in(-\pi,\pi]$ via
\begin{align*}
\cos(\theta_n) &= \vec S_n \cdot \vec S_{n+1} \\
\sin(\theta_{n-1})\sin(\theta_{n}) e^{i\gamma_n} &= (\vec S_{n-1} \times \vec S_{n} )\cdot (\vec S_n \times \vec S_{n+1} ) + i\, \vec S_{n-1} \cdot (\vec S_n \times \vec S_{n+1} ).
\end{align*}
\end{definition}

Note that $\theta_n$ measures the angle between consecutive spins and hence may be considered as a substitute for the curvature appearing in the original Hasimoto transformation.   However, this is not quite the correct choice, as we will see below.  The quantity $\gamma_n$ measures the (signed) angle between the planes spanned by $\{\vec S_{n-1},\vec S_n\}$ and $\{\vec S_{n},\vec S_{n+1}\}$.  As such, it is a natural analogue of the torsion of the curve appearing in the original Hasimoto transform.  We note that while $\gamma_n$ can be regarded as the torsion at site $n$, one should really regard $\theta_n$ as the curvature \emph{between} sites $n$ and $n+1$.  In this sense the coordinates are better seen as being indexed by interlacing lattices, which explains some asymmetry in the formulae that follow.

The functions $(\theta_n,\gamma_n)_{n\in\Z}$ do not form a complete set of coordinates.  Indeed, they are invariant under global rotations:
\begin{equation}
\vec S_n \mapsto \mathcal O \vec S_n  \quad\text{for all $n\in\Z$ and fixed $\mathcal O\in\text{SO}(3)$.}
\end{equation}
This is the only obstruction to inverting this change of coordinates, as is evident from our next lemma.

\begin{lemma}\label{L:recursive}
Given $\vec S_0,\vec S_1\in \mathbb{S}^2$, and $(\theta_n,\gamma_n)_{n\in\Z}$, one can reconstruct the full spin field.  Indeed,
\begin{align*}
\vec S_{n+1} &= \cos(\theta_n) \vec S_n + \tfrac{\sin(\theta_n)}{\sin(\theta_{n-1})} \Bigl[ \sin(\gamma_n) \vec S_{n-1} \times \vec S_n + \cos(\gamma_n) (\vec S_{n-1} \times \vec S_n) \times \vec S_n \Bigr] ,\\
\vec S_{n-1} &= \cos(\theta_{n-1}) \vec S_n + \tfrac{\sin(\theta_{n-1})}{\sin(\theta_{n})}\Bigl[ \sin(\gamma_n) \vec S_{n} \times \vec S_{n+1} - \cos(\gamma_n) (\vec S_{n} \times \vec S_{n+1}) \times \vec S_n \Bigr].
\end{align*}
These relations (and Definition~\ref{D:thetagamma}) also show that
\begin{align*}
\vec S_{n} \cdot \vec S_{n+1} &=\cos(\theta_n),  \\
\vec S_{n} \cdot \vec S_{n+2} &=\cos(\theta_n)\cos(\theta_{n+1}) - \sin(\theta_{n+1})\sin(\theta_{n})\cos(\gamma_{n+1}),  \\
\vec S_{n} \cdot \vec S_{n+3} &=\cos(\theta_{n})\bigl[ \cos(\theta_{n+1})\cos(\theta_{n+2}) - \cos(\gamma_{n+2})\sin(\theta_{n+1})\sin(\theta_{n+2}) \bigr] \\
&\quad +\bigl\{ -\bigl[\sin(\theta_{n+1})\cos(\theta_{n+2})+\cos(\theta_{n+1})\sin(\theta_{n+2})\cos(\gamma_{n+2})\bigr]\cos(\gamma_{n+1}) \\
&\quad +\sin(\theta_{n+2})\sin(\gamma_{n+1})\sin(\gamma_{n+2}) \bigr\}\sin(\theta_n).
\end{align*}
\end{lemma}

\begin{proof}
Note that
\begin{equation}\label{E:frame}
\tfrac{1}{\sin(\theta_{n-1})}(\vec S_{n-1} \times \vec S_n) \times \vec S_n, \quad \tfrac{1}{\sin(\theta_{n-1})} \vec S_{n-1} \times \vec S_n, \qtq{and} \vec S_n
\end{equation}
and
\begin{equation}\label{E:frame'}
\tfrac{1}{\sin(\theta_{n})}(\vec S_{n} \times \vec S_{n+1}) \times \vec S_n, \quad \tfrac{1}{\sin(\theta_{n})} \vec S_{n} \times \vec S_{n+1}, \qtq{and} \vec S_n
\end{equation}
form positively oriented orthonormal bases for $\R^3$.  The first two identities follow by expressing $\vec S_{n+1}$ using \eqref{E:frame} and $\vec S_{n-1}$ using \eqref{E:frame'}.  In particular, the first identity shows that $\theta_n$ and $\gamma_n$ are the traditional spherical polar coordinates for $\vec S_{n+1}$ in this frame.  More precisely, $\theta_n$ represents the colatitude of $\vec S_{n+1}$ relative to a north pole $\vec S_n$.  Analogously, $\gamma_n$ denotes the longitude of $\vec S_{n+1}$ with prime meridian passing through $-\vec S_{n-1}$; this is the sensible choice, since for a slowly varying curve $n\mapsto \vec S_n$, the points $\vec S_{n+1}$ and $\vec S_{n-1}$ will tend to be on opposite sides of $\vec S_{n}$.
\end{proof}

To elucidate the Poisson structure introduced in Definition~\ref{D:pb} at the level of $(\theta_n,\gamma_n)_{n\in\Z}$, we record the following proposition.

\begin{proposition}\label{P:PBthetagamma}
Among the functions $\{\theta_n,\gamma_n : n\in \Z\}$, all non-zero Poisson brackets are as follows:
\begin{center}\def\linespreader{\vphantom{\Big(}}

\vspace*{1ex}

\begin{tabular}{|c||c|}
\hline
$f$            & $\{f,\theta_n\}$\linespreader\\\hline\hline
$\gamma_{n-1}$ & $-\cosec(\theta_{n-1})\cos(\gamma_n)$\linespreader\\\hline
$\theta_{n-1}$ & $\sin(\gamma_{n})$ \linespreader\\\hline
$\gamma_{n}$   & $\cot(\theta_n/2) + \cot(\theta_{n-1})\cos(\gamma_n)$\linespreader\\\hline
$\theta_{n}$   & $0$\linespreader\\\hline
$\gamma_{n+1}$ & $-\cot(\theta_n/2) - \cot(\theta_{n+1})\cos(\gamma_{n+1})$\linespreader\\\hline
$\theta_{n+1}$ & $-\sin(\gamma_{n+1})$ \linespreader\\\hline
$\gamma_{n+2}$ & $\cosec(\theta_{n+1})\cos(\gamma_{n+1})$ \linespreader\\\hline
\end{tabular}

\vspace*{2ex}

\begin{tabular}{|c||c|}
\hline
$f$            & $\{f,\gamma_n\}$\linespreader\\\hline\hline
$\gamma_{n-2}$ & $-{\sin(\gamma_{n-1})}{\cosec(\theta_{n-2})\cosec(\theta_{n-1})}$\linespreader\\\hline
$\gamma_{n-1}$ & $\bigl[{\cot(\theta_{n-2})\sin(\gamma_{n-1})+\cot(\theta_{n})\sin(\gamma_{n})}\bigr]{\cosec(\theta_{n-1})}$\linespreader\\\hline
$\gamma_{n}$   & $0$\linespreader\\\hline
$\gamma_{n+1}$ & $-\bigl[{\cot(\theta_{n-1})\sin(\gamma_{n})+\cot(\theta_{n+1})\sin(\gamma_{n+1})}\bigr]{\cosec(\theta_{n})}$\linespreader\\\hline
$\gamma_{n+2}$ & ${\sin(\gamma_{n+1})}{\cosec(\theta_{n})\cosec(\theta_{n+1})}$ \linespreader\\\hline
\end{tabular}

\vspace*{1ex}

\end{center}
together with those determined by the above via anti-symmetry.
\end{proposition}

\begin{proof}
The exact calculations are lengthy; we summarize the method, rather than give all details.

Using Definitions~\ref{D:pb} and \ref{D:thetagamma}, it is easy to compute
\begin{align*}
\bigl\{ \vec S_m & \cdot \vec S_{m+1} ,\, \vec S_n \cdot \vec S_{n+1} \bigr\} \\
&= \delta_{m,n+1} \vec S_{n+2}\cdot (\vec S_{n+1} \times \vec S_n)-\delta_{m,n-1} \vec S_{m+2}\cdot (\vec S_{m+1} \times \vec S_m)\\
&=- \delta_{m,n+1} \sin(\theta_n)\sin(\theta_{n+1})\sin(\gamma_{n+1})+\delta_{m,n-1} \sin(\theta_m)\sin(\theta_{m+1})\sin(\gamma_{m+1}).
\end{align*}
On the other hand,
$$
\bigl\{ \vec S_m \cdot \vec S_{m+1} ,\, \vec S_n \cdot \vec S_{n+1} \bigr\} = \bigl\{ \cos(\theta_m),\, \cos(\theta_n) \bigr\} = \sin(\theta_m)\sin(\theta_n)\{\theta_m,\,\theta_n\}.
$$
This yields all Poisson brackets of the form $\{\theta_m,\theta_n\}$.

By the Jacobi identity and the previous result,
\begin{align*}
\cos(\gamma_m)\{\gamma_m,\theta_n\}=\{\sin(\gamma_m),\theta_n\} &=  \{\{\theta_{m-1},\theta_m\},\theta_n\} \\
&= - \{\{\theta_m,\theta_n\},\theta_{m-1}\} - \{\{\theta_n,\theta_{m-1}\},\theta_m\},
\end{align*}
which shows (using the previous result again) that this quantity is zero unless $m\in\{n-1,n,n+1,n+2\}$. To actually determine the values in these four cases, we compute
$$
\bigl\{ \vec S_{m-1} \cdot \bigl(\vec S_m \times \vec S_{m+1}\bigr) ,\, \vec S_n \cdot \vec S_{n+1} \bigr\} = \bigl\{ \sin(\theta_{m-1})\sin(\theta_m)\sin(\gamma_m),\, \cos(\theta_n) \bigr\}
$$
directly from Definition~\ref{D:pb}.  As the example
$$
\bigl\{ \vec S_{n-2} \cdot \bigl(\vec S_{n-1} \times \vec S_{n}\bigr) ,\, \vec S_n \cdot \vec S_{n+1} \bigr\}
    = \bigl(\vec S_{n-2}\cdot\vec S_{n}\bigr)\bigl(\vec S_{n-1}\cdot\vec S_{n+1}\bigr) - \bigl(\vec S_{n-1}\cdot\vec S_{n}\bigr)\bigl(\vec S_{n-2}\cdot\vec S_{n+1}\bigr)
$$
shows, this requires expressing various dot products in terms of $\theta$ and $\gamma$.  This is possible through applications of Lemma~\ref{L:recursive}.  Performing these computations yields all the information presented in the first table.

Arguing as previously, we have
\begin{align*}
\{\sin(\gamma_m),\sin(\gamma_n)\} &= \{ \{\theta_{m-1},\theta_m\},\sin(\gamma_n)\} \\
&= \{ \{\sin(\gamma_n),\theta_m\},\theta_{m-1}\} - \{ \{\sin(\gamma_n),\theta_{m-1}\},\theta_m\}.
\end{align*}
Thus the values shown in the second table can be deduced from those in the first, with only the expenditure of sufficient labour.
\end{proof}

\begin{definition}[Discrete Hasimoto transform]\label{D:alpha}
For a field $\vec S:\Z\to \mathbb{S}^2$, we define complex coordinates $\alpha_n\in \C$ via
\begin{align}\label{E:D:alpha}
\alpha_n = \tan(\theta_n/2) e^{-i\Gamma(n)} \qtq{where} \Gamma(n) := \sum_{\ell\leq n} \gamma_\ell
\end{align}
and $\theta_n\in(0,\pi)$ and $\gamma_n\in (-\pi,\pi]$ are as in Definition~\ref{D:thetagamma}.
\end{definition}

Included in this definition is the assertion that $\tan(\theta_n/2)$ is the proper discrete analogue of the curvature in \eqref{Hasi}.  Unaware that it appears already in \cite[equation (14a)]{Ishimori}, we originally  intuited this relation by comparing conserved quantities for \eqref{E:S} and \eqref{E:AL}; see \eqref{HLHMtoAL} below.

The domain of the functions $\alpha_n$ is a rather thin set within all possible spin configurations.  Not only must we avoid consecutive spins being parallel or anti-parallel, but we must now also constrain the torsion $\gamma_n$ to be summable.  Below we will determine the Poisson brackets of these functions of the spins and find that the results are polynomials in these same functions.  This induces a Poisson structure on the algebra of finitely supported smooth functions of the variables $\alpha_n$, which may now be regarded as an independent object, free from the constraints just mentioned.  From this perspective, one may simply take the results of Proposition~\ref{P:PBa} as the definition of a Poisson structure on such an algebra, which happens to be inspired by the spin model.  However, before one simply accepts the formulae below as the definition of a Possion structure, one must verify the Jacobi identity.

While it is indeed elementary (though tedious) to verify the Jacobi identity directly --- indeed, we did this as a check on our computations --- this is unnecessary since the domain of the functions $\alpha_n$ is nonetheless rich enough to guarantee that this identity is inherited from the corresponding relation for \eqref{PB}.

\begin{proposition}\label{P:PBa}
Poisson brackets among the functions $\{\Re\alpha_n,\Im\alpha_n : n\in \Z\}$ are as follows
\begin{align*}
\{ \Re \alpha_n, \Im \alpha_m\}=
\begin{cases}
-\tfrac{1+|\alpha_m|^2}2 \Im \alpha_n \Im(\alpha_{m-1}-\alpha_{m+1}), & n\geq m+2,\\
-\tfrac{1+|\alpha_m|^2}2 \Bigl[ \Im \alpha_{n} \Im(\alpha_{m-1}-\alpha_{m+1}) + \tfrac{1+|\alpha_{n}|^2}2\Bigr], & n= m+1,\\
\tfrac{1+|\alpha_n|^2}2-\tfrac{1+|\alpha_n|^2}2\Re(\alpha_n\overline{\alpha_{n-1}}), & n=m, \\
-\tfrac{1+|\alpha_n|^2}2 \Bigl[\Re \alpha_{m} \Re(\alpha_{n-1}-\alpha_{n+1}) + \tfrac{1+|\alpha_{m}|^2}2\Bigr], & n= m-1,\\
-\tfrac{1+|\alpha_n|^2}2 \Re \alpha_m \Re(\alpha_{n-1}-\alpha_{n+1}), & n\leq m-2,
\end{cases}
\end{align*}
and
\begin{align*}
\{ \Re \alpha_n, \Re \alpha_m\} &= -\tfrac{1+|\alpha_m|^2}2 \Im \alpha_n \Re(\alpha_{m-1}-\alpha_{m+1}),\qtq{for} n\geq m+1, \\
\{ \Im \alpha_n, \Im \alpha_m\} &= \tfrac{1+|\alpha_m|^2}2 \Re \alpha_n \Im(\alpha_{m-1}-\alpha_{m+1}),\phantom{+}\qtq{for} n\geq m+1.
\end{align*}
These determine all remaining cases through anti-symmetry.
\end{proposition}

\begin{proof}
Using Proposition~\ref{P:PBthetagamma}, it is elementary to verify that
\begin{align*}
\{ \Gamma(n), \theta_k\} = \begin{cases}
-\tan(\theta_{k-1}/2)\cos(\gamma_k)+\tan(\theta_{k+1}/2)\cos(\gamma_{k+1}),&n\geq k+2,\\
-\tan(\theta_{k-1}/2)\cos(\gamma_k)-\cot(\theta_{k+1})\cos(\gamma_{k+1}),&n= k+1,\\
-\tan(\theta_{k-1}/2)\cos(\gamma_k)+\cot(\theta_k/2),&n=k,\\
-\cosec(\theta_{k-1})\cos(\gamma_k),&n= k-1,\\
0, &n\leq k-2.\end{cases}
\end{align*}

To complete the calculations, we also need to know $\{ \Gamma(n), \Gamma(m)\}$ for all $n$ and~$m$.  Due to the finite-range nature of the Poisson bracket detailed in Proposition~\ref{P:PBthetagamma}, these are easily determined.  Indeed,
\begin{align*}
\{ \Gamma(m+1), \Gamma(m)\} &= \{ \Gamma(m+1)-\Gamma(m), \Gamma(m)\} = \{ \gamma_{m+1}, \gamma_m+\gamma_{m-1}\} \\
&= \bigl[\tan(\theta_{m-1}/2)\sin(\gamma_{m})-\cot(\theta_{m+1})\sin(\gamma_{m+1})\bigr]\cosec(\theta_{m}).
\end{align*}
Similarly, for $n\geq m+2$, we have
\begin{align*}
\{ \Gamma(n), \Gamma(m)\} =  \bigl[\tan(\theta_{m-1}/2)\sin(\gamma_{m})+\tan(\theta_{m+1}/2)\sin(\gamma_{m+1})\bigr]\cosec(\theta_{m}).
\end{align*}
These determine all other cases via antisymmetry.
\end{proof}

Using the new coordinates, we can rewrite the Hamiltonian \eqref{HS} as
\begin{align}\label{HLHMtoAL}
\HS =  \sum_n  4\log\bigl[\sec\bigl(\tfrac{\theta_n}2\bigr)\bigr] = \sum_n 2\log\bigl(1+|\alpha_n|^2\bigr).
\end{align}
This is the discrete analogue of \eqref{HisM}.  The right-hand side here is a well-known conservation law in the context of the Ablowitz--Ladik system, where it plays the role analogous to that played by the mass for the NLS equation.  Concretely, for solutions to \eqref{E:AL}, we have
$$
\partial_t \log\bigl(1+|\alpha_n|^2\bigr)  = - 2\Im\bigl(\bar\alpha_n \alpha_{n+1}\bigr) + 2\Im\bigl(\bar\alpha_{n-1} \alpha_{n}\bigr).
$$

As mentioned before, we initially derived \eqref{E:D:alpha} by finding what relation between $\theta_n$ and $|\alpha_n|$ was necessary to arrive at the identity \eqref{HLHMtoAL}.

For comparison, the Hamiltonian corresponding to the Heisenberg spin chain model \eqref{Heis} becomes
\begin{align*}
\HH = \sum_n  2\sin^2\bigl(\tfrac{\theta_n}2\bigr) = \sum_n \tfrac{2|\alpha_n|^2}{1+|\alpha_n|^2}.
\end{align*}

\begin{lemma}\label{L:elltwo}
Consider the phase space $\ell^2(\Z)$ endowed with the Poisson bracket laid out in Proposition~\ref{P:PBa}. The Hamiltonian \eqref{HLHMtoAL} induces the focusing Ablowitz--Ladik flow \eqref{E:AL}, which is globally wellposed.
\end{lemma}

\begin{proof}
It is evident that the infinite sum \eqref{HLHMtoAL} converges for $\alpha\in\ell^2(\Z)$.  Moreover, from Proposition~\ref{P:PBa}, we have
\begin{align*}
 i \bigl\{ \alpha_n , 2\log( 1+|\alpha_k|^2 )\bigr\} = \begin{cases}
-2\Re\bigl[\bar\alpha_k (\alpha_{k-1}-\alpha_{k+1})\bigr]\alpha_n, &n\geq k+2 \\
-2\Re\bigl[\bar\alpha_k (\alpha_{k-1}-\alpha_{k+1})\bigr]\alpha_n-( 1+|\alpha_n|^2 )\alpha_k, &n=k+1 \\
-2\Re\bigl[\bar\alpha_k \alpha_{k-1}\bigr]\alpha_n + 2\alpha_n, &n=k \\
-( 1+|\alpha_n|^2 )\alpha_k, &n=k-1 \\
0, & n\leq k-2\end{cases}
\end{align*}
which shows that the induced vector fields are also summable, yielding
\begin{equation}\label{E:ALagain}
i\partial_t \alpha_n = \sum_k i \bigl\{ \alpha_n , 2\log( 1+|\alpha_k|^2 )\bigr\} = -\bigl(1+ |\alpha_n|^2 \bigr)  \bigl[\alpha_{n+1}+\alpha_{n-1}\bigr] + 2\alpha_n
\end{equation}
which is the Ablowitz--Ladik flow \eqref{E:AL}.

The local well-posedness of \eqref{E:ALagain} is trivial, since RHS\eqref{E:ALagain} defines a locally Lipschitz vector field on $\ell^2(\Z)$.  This extends to global well-posedness due to conservation of the Hamiltonian \eqref{HLHMtoAL}, which controls the $\ell^2$ norm.
\end{proof}

While the context in which we derived Lemma~\ref{L:elltwo} explains the connection of the Ablowitz--Ladik equation to \eqref{E:S}, it does little to help us understand invariant measures.  We would like to truncate in space, obtain invariant measures in that setting, and then pass to the infinite volume limit.  Such spatial truncations are rather violently at odds with the infinite-range character of the Poisson structure given in Proposition~\ref{P:PBa}.

Secondly, the traditional construction of invariant measures in Hamiltonian mechanics rests on the invariance of phase volume (Liouville's Theorem).  It is far from clear what phase volume we should associate with the Poisson structure we have studied thus far.

The remedy to both our troubles lies in the fact that the Ablowitz--Ladik equation is bi-Hamiltonian (in the sense of \cite{Magri}), as we will explain.  Let us begin by recalling the standard Hamiltonian formulation of the Ablowitz--Ladik equation, as laid out in \cite{FT:book}, for example.

\begin{definition}\label{D:pb0}
We define a second Poisson structure on the algebra generated by $\{\Re\alpha_n,\Im\alpha_n : n\in \Z\}$ as follows:
$$
\bigl\{ \Re \alpha_n, \Im \alpha_m \bigr\}_0 = - \bigl\{ \Im \alpha_n, \Re \alpha_m \bigr\}_0 = (1+|\alpha_n|^2) \delta_{nm}
$$
and all other brackets are zero.
\end{definition}

We note that this corresponds the symplectic structure
\begin{equation}\label{omega zero}
\omega_0 = \sum_{n\in\Z} (1+|\alpha_n|^2)^{-1} d\Re(\alpha_n)\wedge d\Im(\alpha_n)
\end{equation}
and that the flow \eqref{E:AL} is generated by
\begin{equation}\label{H AL}
H_\text{AL} := \sum_{n\in\Z} -\Re(\bar\alpha_n\alpha_{n+1})+\log(1+|\alpha_n|^2),
\end{equation}
which Poisson commutes with $\HS$.

While this shows that the Ablowitz--Ladik equation admits a second Hamiltonian interpretation, this is slightly less than being bi-Hamiltonian.  One needs to show that the two Poisson structures are compatible, namely, that any linear combination of the two Poisson brackets remains a Poisson bracket.   The only obstruction to compatibility is the Jacobi identity.

\begin{theorem}
The Poisson brackets of Proposition~\ref{P:PBa} and Definition~\ref{D:pb0} are compatible.
\end{theorem}

\begin{proof}
As we already know that each of the Poisson brackets obeys the Jacobi identity individually, it suffices to show that
$$
\sum \{F,\{G,H\}_0\} +\{F,\{G,H\}\}_0 =0,
$$
where the sum is taken over the three cyclic permutations of the functions $F$, $G$, and $H$.  Moreover, it suffices to select each of these three functions from the collection $\{\Re \alpha_n, \Im\alpha_n : n\in\Z\}$.  Due to the zero-range structure of the $\{\,,\}_0$ bracket, these observations reduce matters to a finite collection of computations that one simply has to grind through.  As a finite system of polynomial identities, this is also amenable to checking via computer algebra systems.
\end{proof}

While the existence of multiple Hamiltonian interpretations of the Ablowitz--Ladik system has been know for some time (see \cite{Lozano} and references therein), to the best of our knowledge no previous authors have verified compatibility; see, for example, \cite[\S5]{ErcLozano}.

As described earlier, our interest in this alternate Poisson structure stems from the problem of constructing invariant measures for truncations of the system.

We obtain our finite-volume model by truncating the Hamiltonian \eqref{H AL}: Given an integer $K>0$,
\begin{align}\label{AL H free}
H_\text{AL}^K :=  \sum_{n=-K}^{K-1} -\Re(\bar\alpha_n\alpha_{n+1})+ \sum_{n=-K}^{K}\log(1+|\alpha_n|^2)
\end{align}
generates the following dynamics
\begin{equation}\label{AL free}
i\tfrac{d\ }{dt} \alpha_n = \{\alpha_n,H_\text{AL}^K \}_0 =
\begin{cases}
-\bigl(1+ |\alpha_{-K}|^2 \bigr)  \alpha_{-K+1}+ 2\alpha_{-K},       &n=-K,\\
-\bigl(1+ |\alpha_n|^2 \bigr)  \bigl[\alpha_{n+1}+\alpha_{n-1}\bigr] + 2\alpha_n, &|n|\leq K-1,\\
-\bigl(1+ |\alpha_K|^2 \bigr)  \alpha_{K-1}+ 2\alpha_K,              &n=K,
\end{cases}
\end{equation}
which is easily seen to conserve
\begin{align}\label{H FT free}
\HS^K :=  \sum_{|n|\leq K}  4\log\bigl[\sec\bigl(\tfrac{\theta_n}2\bigr)\bigr] = \sum_{|n|\leq K} 2\log\bigl(1+|\alpha_n|^2\bigr).
\end{align}

At the level of the spins, $\HS^K$ is the energy functional corresponding to free boundary conditions --- the spins at the ends of the chain only couple to their one neighbour.  One could also consider other boundary conditions. However, we will prove uniqueness of both the Gibbs measure and the dynamics in infinite volume; thus, the choice of boundary condition has no effect.

\begin{proposition}\label{P:trunc} The truncated Ablowitz--Ladik system \eqref{AL free} is globally wellposed and conserves the following `white-noise' probability measure
\begin{equation}
d\mu_\textit{wn}^{\beta,K} = \prod_{-K\leq n\leq K}\frac{1+2\beta}{\pi} \frac {d\text{\rm Area}(\alpha_n)}{(1+|\alpha_n|^2)^{2+2\beta}}
\end{equation}
for any $\beta>0$.
\end{proposition}

\begin{proof}
As RHS\eqref{AL free} is a Lipschitz function on $\C^{2K+1}$, local well-posedness follows immediately.  This can be made global in time due to conservation of the coercive quantity \eqref{H FT free}.

By writing
\begin{equation}\label{E:wn not Gibbs}
d\mu_\textit{wn}^{\beta,K} = \bigl(\tfrac{1+2\beta}{\pi}\bigr)^{2K+1}  e^{-(\beta+\frac12)\HS^K } \prod_n \frac{d\Re(\alpha_n)\wedge d\Im(\alpha_n)}{(1+|\alpha_n|^2)},
\end{equation}
we see that the preservation of this measure under the flow stems from conservation of $\HS^K$ and Liouville's Theorem on the preservation of phase volume (cf. \eqref{omega zero}).
\end{proof}

We note that \eqref{E:wn not Gibbs} deviates rather sharply from the Gibbs measure one would naturally associate with the system \eqref{AL free}: the inverse temperature is shifted and multiplies the analogue of mass, rather than the Hamiltonian. These anomalies will disappear when we pass back through the discrete Hasimoto transform --- we will see that under this correspondence, this measure does indeed map to the true Gibbs measure for the spin system.
These anomalies also serve to remind us of the subtle interrelation between the two Hamiltonian structures.

%%%%%%%%%%%%%%%%%%%%%%%%%%%%%%%%%%%%%%%%%%%%%%%%%%%%%%%%%
\section{The discrete Hasimoto transform via parallel frames}\label{S:AL to S}
%%%%%%%%%%%%%%%%%%%%%%%%%%%%%%%%%%%%%%%%%%%%%%%%%%%%%%%%%

In this section we revisit the discrete Hasimoto transform from the modern perspective of parallel frames.  In order to complete the program laid out in the introduction, we will need to show how to transfer solutions from the Ablowitz--Ladik system to the spin chain model.  This is the major impetus of this section; see Theorem~\ref{T:AL to S}.  We start by introducing some notation.  For $z\in \C$ we define the orthogonal matrix
\begin{equation}\label{Q}
\begin{aligned}
Q(z)
&= \frac{1}{1+|z|^2}\begin{bmatrix} 1-\Re(z^2) &  \Im (z^2) &  2\Re(z)\\[1ex]
            \Im(z^2) & 1+\Re(z^2) & -2\Im(z)\\[1ex]
            -2\Re(z) & 2\Im(z) & 1-|z|^2 \end{bmatrix}.
\end{aligned}
\end{equation}
Note that $Q(z)$ is the exponential of the antisymmetric matrix
\begin{equation}\label{q}
\begin{aligned}
q(z)=\begin{bmatrix} 0 &  0 &  2\arctan(|z|)\frac{\Re(z)}{|z|}\\[1ex]
            0 & 0 & -2\arctan(|z|)\frac{\Im(z)}{|z|}\\[1ex]
            -2\arctan(|z|)\frac{\Re(z)}{|z|} & 2\arctan(|z|)\frac{\Im(z)}{|z|} &0 \end{bmatrix}.
\end{aligned}
\end{equation}

\begin{proposition}\label{P:Bishop Hasimoto}
Let $\{\vec S_n\}_{n\in\Z}$ be a sequence of spins such that no two consecutive spins are antiparallel.  Let $P_0\in \SO(3)$ be such that
\begin{equation*}
\begin{aligned}
\vec S_0=P_0 \vec{e_3}
\end{aligned}.
\end{equation*}
Then there exists a unique sequence $\{\alpha_n\}_{n\in\Z}$ of complex numbers such that with
\begin{equation}\label{frame update}
Q_n=Q(\alpha_n) \qtq{and} P_{n+1}=P_nQ_n
\end{equation}
we have
\begin{equation}\label{1:48}
\vec S_n=P_n\vec{e_3}.
\end{equation}
Moreover, for all $n\in\Z$ we have
\begin{align}
\vec S_n \cdot \vec S_{n+1} &= \frac{1-|\alpha_n|^2}{1+|\alpha_n|^2} ,\label{FB consistent1}\\
\!\!\!(\vec S_{n-1} \times \vec S_{n} )\cdot (\vec S_n \times \vec S_{n+1} ) + i\, \vec S_{n-1} \cdot (\vec S_n \times \vec S_{n+1} )&= \frac{4\bar{\alpha}_n\alpha_{n-1}}{(1+|\alpha_n|^2)(1+|\alpha_{n-1}|^2)},\label{FB consistent2}
\end{align}
from which we see that the map $\{\vec S_n\}_{n\in\Z} \mapsto\{\alpha_n\}_{n\in\Z}$ agrees with the one constructed in Section~\ref{S:DH} modulo $U(1)$ gauge invariance.
\end{proposition}

Before turning to the proof of this proposition, let us first explain the sense in which it encapsulates the modern approach to the Hasimoto transform via parallel frames. As $P_n\in SO(3)$, its columns form a positively oriented orthonormal basis for $\R^3$.  By \eqref{1:48}, the third column coincides with $\vec S_n$, which in the context of the original Hasimoto transform means that it is tangent to the vortex curve.  The remaining two columns form an othonormal basis normal to the curve.

In the continuum setting, one asks that the derivatives of these normal vectors along the curve be parallel to the tangent to the curve, that is, they are given by parallel transport.  Equivalently, the frame $P:\R\to SO(3)$ obeys
\begin{equation}\label{cont parallel}
\partial_x P = A P \qtq{where} A =\begin{bmatrix} 0 & 0 & \kappa_1(x)\\ 0 & 0 & \kappa_2(x)\\ -\kappa_1(x) & -\kappa_2(x) & 0 \end{bmatrix}
\end{equation}
and $\kappa_1,\kappa_2$ are functions (dictated by the geometry of the curve) that ensure $P(x)\vec e_3$ remains tangent to the curve.  It is not difficult to verify that the modulus of $\kappa_1+i\kappa_2$ coincides with the curvature of the curve, while the derivative of its argument is the torsion of the curve; see \cite{Bishop,Peter} for details.  Comparing with \eqref{Hasi}, we see that $\kappa_1(x)+i\kappa_2(x) = \bar\psi(x)$ modulo a global phase rotation.

Let us now compare the continuum setup with that of Proposition~\ref{P:Bishop Hasimoto}.  First we see that the distribution of non-zero entries in $A$ matches that in $q(z)$ given above; moreover, matching the non-zero entries in $A$ to those in $q(\alpha_n)$ leads via \eqref{E:D:alpha} to the relation $\kappa_1+i\kappa_2 = \theta_n e^{i\Gamma(n)}$, which matches the continuum analogue.  This further explains the appearance of the tangent function in \eqref{E:D:alpha}.

\begin{proof}[Proof of Proposition~\ref{P:Bishop Hasimoto}]
The key observation is that
$$
z\mapsto Q(z) \vec{e_3}=\frac{1}{1+|z|^2}\begin{bmatrix}2\Re(z)\\[1ex]-2\Im(z)\\[1ex]1-|z|^2 \end{bmatrix}
$$
maps $\C$ bijectively onto $\mathbb{S}^2\setminus\{-\vec e_3\}$; indeed it is essentially the inverse of the stereographic projection.  As $\vec S_0 \cdot\vec S_1\neq -1$, it follows that there exists a unique $\alpha_0\in \C$ such that
$$
P_0^T\vec S_1 =Q(\alpha_0) \vec e_3 \qtq{or equivalently,} \vec S_1 = P_0 Q(\alpha_0) \vec{e_3}.
$$
Using this observation and arguing inductively, one easily constructs uniquely the remaining $\alpha_n$ such that \eqref{1:48} holds.  It remains to verify \eqref{FB consistent1} and \eqref{FB consistent2}.

Using that $P_n$ is an orthogonal matrix, we get
$$
\vec S_n \cdot \vec S_{n+1} = P_n \vec{e_3}\cdot P_nQ_n \vec{e_3}=\vec{e_3}\cdot Q_n \vec{e_3} =\frac{1-|\alpha_n|^2}{1+|\alpha_n|^2},
$$
which is \eqref{FB consistent1}.

To continue, we use the fact that for any matrix $\mathcal O\in SO(3)$ and any vector $\vec v$,
\begin{equation}\label{trick}
\begin{aligned}
\bigl(\mathcal O \vec e_3 \bigr)\times \vec v = \mathcal O \begin{bmatrix}  0 & -1 & 0\\[1ex] 1 & 0 & 0\\[1ex] 0 & 0 & 0
\end{bmatrix} \mathcal O^T \vec v.
\end{aligned}
\end{equation}
This allows us to compute
\begin{equation*}
\begin{aligned}
\vec S_n \times \vec S_{n+1}
&= P_n \begin{bmatrix}  0 & -1 & 0\\[1ex] 1 & 0 & 0\\[1ex] 0 & 0 & 0\end{bmatrix} P_n^T P_n Q_n \vec{e_3}= \frac1{1+|\alpha_n|^2}P_n \begin{bmatrix}  2\Im(\alpha_n)\\[1ex] 2\Re(\alpha_n)\\[1ex] 0\end{bmatrix}.
\end{aligned}
\end{equation*}
Thus,
\begin{align}\label{1:56}
\vec S_{n-1} \cdot (\vec S_n \times \vec S_{n+1} )
&= P_nQ_{n-1}^T\vec{e_3} \cdot\frac1{1+|\alpha_n|^2}P_n \begin{bmatrix}  2\Im(\alpha_n)\\[1ex] 2\Re(\alpha_n)\\[1ex] 0\end{bmatrix} \notag\\
&= \frac{4\Im[\bar{\alpha}_n\alpha_{n-1}]}{(1+|\alpha_n|^2)(1+|\alpha_{n-1}|^2)}.
\end{align}
Using also \eqref{FB consistent1}, we get
\begin{align}\label{1:57}
(\vec S_{n-1} \times \vec S_{n} )\cdot (\vec S_n \times \vec S_{n+1} )&= (\vec S_{n-1}\cdot \vec S_n) (\vec S_n \cdot \vec S_{n+1} ) - \vec S_{n-1} \cdot \vec S_{n+1}\notag\\
&=(\vec S_{n-1}\cdot \vec S_n) (\vec S_n \cdot \vec S_{n+1} ) -P_{n-1} \vec{e_3}\cdot P_{n-1}Q_{n-1}Q_n \vec{e_3}\notag\\
&=\frac{(1-|\alpha_{n-1}|^2)(1-|\alpha_n|^2)}{(1+|\alpha_{n-1}|^2)(1+|\alpha_n|^2)} -Q_{n-1}^T \vec{e_3}\cdot Q_n \vec{e_3}\notag\\
&= \frac{4\Re[\bar{\alpha}_n\alpha_{n-1}]}{(1+|\alpha_n|^2)(1+|\alpha_{n-1}|^2)}.
\end{align}
Collecting \eqref{1:56} and \eqref{1:57}, we obtain \eqref{FB consistent2}.
\end{proof}

As announced earlier, the main goal of this section is to `invert' the discrete Hasimoto transform.  To this end, let $\alpha:\Z\times\R\to \C$ be a solution to the Ablowitz--Ladik system \eqref{E:AL}.  For $n\in \Z$, we define
\begin{align}\label{Qn}
Q_n(t)= Q(\alpha_n(t))
\end{align}
and
\begin{equation}\label{An}
\begin{aligned}
A_n(t)
&= \begin{bmatrix} 0 & -2\Re(\bar{\alpha}_n\alpha_{n-1}) & -2\Im (\alpha_n-\alpha_{n-1})\\[1ex]
            2\Re(\bar{\alpha}_n\alpha_{n-1})& 0 & -2\Re (\alpha_n-\alpha_{n-1})\\[1ex]
            2\Im (\alpha_n-\alpha_{n-1}) & 2\Re (\alpha_n-\alpha_{n-1}) & 0\end{bmatrix}.
\end{aligned}
\end{equation}

Fix $\mathcal O\in \SO(3)$ and let $P_0(t)$ be the solution to the initial-value problem
\begin{equation}\label{P0}
\frac{d}{dt} P_0 = P_0 A_0 \qtq{with} P_0(t=0)=\mathcal O.
\end{equation}
For all other $n\in \Z\setminus\{0\}$, we define $P_n(t)$ via the recurrence relation
\begin{align}\label{Pn}
P_{n+1}(t)= P_n(t) Q_n(t).
\end{align}

\begin{lemma}\label{L:QP}
Assume $\alpha:\Z\times\R\to \C$ is a solution to the Ablowitz--Ladik system \eqref{E:AL}.  Let $\{Q_n\}_{n\in \Z}$ and $\{P_n\}_{n\in \Z}$ be as defined by \eqref{Qn} through \eqref{Pn}.  Then for all $n\in\Z$, we have
\begin{align}
\frac{d}{dt} Q_n &= Q_n A_{n+1} - A_n Q_n \label{deriv Q}\\
\frac{d}{dt} P_n &= P_n  A_n \label{deriv P}.
\end{align}
\end{lemma}

\begin{remark}
The identity \eqref{deriv Q} can be interpreted as an $SO(3)$-valued zero-curvature representation of the Ablowitz--Ladik model.  The usual $2\times2$ representation (cf. \cite{AL}) is inferior for our purposes since it leads to a less transparent action of the $SO(3)$ gauge group of the spin chain model.
\end{remark}

\begin{proof}
The claim \eqref{deriv Q} follows from a lengthy computation, using \eqref{E:AL} to compute the time derivative of $Q_n$.  We omit the details.

To prove \eqref{deriv P}, we argue by induction.  For $n=0$, \eqref{deriv P} is precisely the definition of $P_0$.  Assuming \eqref{deriv P} holds for some $n\geq 0$, and using \eqref{Pn} and \eqref{deriv Q}, we compute
\begin{align*}
P_{n+1}^T\tfrac{d}{dt} P_{n+1}
&= Q_n^T P_n^T\Bigl[ \bigl(\tfrac{d}{dt} P_n\bigr)Q_n +P_n\tfrac{d}{dt} Q_n \Bigr]\\
&= Q_n^T P_n^T\bigl[P_n  A_n Q_n + P_n (Q_n A_{n+1} -A_n Q_n)\bigr]\\
&=A_{n+1}.
\end{align*}
Similarly, assuming that \eqref{deriv P} holds for some $n+1\leq 0$, and using \eqref{Pn}, \eqref{deriv Q}, and the fact that the matrices $A_n$ are antisymmetric, we compute
\begin{align*}
P_n^T\tfrac{d}{dt} P_n
&= \bigl(P_{n+1}Q_n^T\bigr)^T \tfrac{d}{dt} \bigl(P_{n+1}Q_n^T\bigr)\\
&=Q_n P_{n+1}^T\Bigl[ \bigl(\tfrac{d}{dt} P_{n+1}\bigr)Q_n^T +P_{n+1}\tfrac{d}{dt} Q_n^T \Bigr]\\
&= Q_n P_{n+1}^T\bigl[P_{n+1}  A_{n+1} Q_n^T + P_{n+1} (A_{n+1}^TQ_n^T - Q_n^TA_n^T)\bigr]\\
&=Q_n (A_{n+1}+ A_{n+1}^T) Q_n^T - A_n^T \\
&= A_n.
\end{align*}
This completes the proof of the lemma.
\end{proof}

\begin{theorem}\label{T:AL to S}
Let $\mathcal O\in \SO(3)$ and let $\alpha:\Z\times\R\to \C$ be a solution to the Ablowitz--Ladik system \eqref{E:AL}.  Let $\{Q_n\}_{n\in \Z}$ and $\{P_n\}_{n\in \Z}$ be as defined by \eqref{Qn} through \eqref{Pn}.  Then $\vec S:\Z\times\R\to \mathbb{S}^2$ given by
\begin{equation}\label{S def}
\begin{aligned}
\vec S_n (t)=P_n(t)\vec{e_3}
\end{aligned}
\end{equation}
is a solution to the system \eqref{E:S}.
\end{theorem}

\begin{proof}
On one hand, using Lemma~\ref{L:QP}, we get
\begin{equation}\label{2:37}
\begin{aligned}
\frac{d}{dt}\vec S_n = \frac{d}{dt}P_n\vec{e_3} = P_n A_n \vec{e_3} = P_n \begin{bmatrix}  -2\Im(\alpha_n -\alpha_{n-1})\\[1ex] -2\Re(\alpha_n -\alpha_{n-1})\\[1ex] 0
\end{bmatrix}.
\end{aligned}
\end{equation}

On the other hand, using \eqref{Pn} we compute
\begin{equation*}
\begin{aligned}
1+\vec S_n \cdot \vec S_{n+1}
= 1+ P_n\vec{e_3}\cdot P_nQ_n\vec{e_3}
= 1+\frac{1-|\alpha_n|^2}{1+|\alpha_n|^2}= \frac2{1+|\alpha_n|^2}.
\end{aligned}
\end{equation*}
Using also \eqref{trick}, we find
\begin{equation*}
\begin{aligned}
\vec S_n \times \vec S_{n+1}
&= P_n \begin{bmatrix}  0 & -1 & 0\\[1ex] 1 & 0 & 0\\[1ex] 0 & 0 & 0\end{bmatrix} P_n^T P_n Q_n \vec{e_3}
= \frac1{1+|\alpha_n|^2}P_n \begin{bmatrix}  2\Im(\alpha_n)\\[1ex] 2\Re(\alpha_n)\\[1ex] 0\end{bmatrix}.
\end{aligned}
\end{equation*}
Thus,
\begin{equation*}
\begin{aligned}
-  \frac{2\vec S_n\times\vec S_{n+1}}{1+\vec S_n\cdot\vec S_{n+1}} - \frac{2\vec S_n\times\vec S_{n-1}}{1+\vec S_n\cdot\vec S_{n-1}}
%&= - P_n \begin{bmatrix}  2\Im(\alpha_n)\\[1ex] 2\Re(\alpha_n)\\[1ex] 0\end{bmatrix} +P_{n-1}\begin{bmatrix}  2\Im(\alpha_{n-1})\\[1ex] 2\Re(\alpha_{n-1})\\[1ex] 0\end{bmatrix}\\
&= P_n\begin{bmatrix}  -2\Im(\alpha_n)\\[1ex] -2\Re(\alpha_n)\\[1ex] 0\end{bmatrix} +P_n Q_{n-1}^T\begin{bmatrix}  2\Im(\alpha_{n-1})\\[1ex] 2\Re(\alpha_{n-1})\\[1ex] 0\end{bmatrix}.
\end{aligned}
\end{equation*}

It is easy to verify that for each $n\in \Z$, the vector $\begin{bmatrix}  2\Im(\alpha_{n}) & 2\Re(\alpha_{n}) & 0\end{bmatrix}^T$ is an eigenvector for $Q_{n}$ with eigenvalue $1$.  Indeed, this vector belongs to the kernel of $q(\alpha_n)$, where $q$ is the antisymmetric matrix defined in \eqref{q}.  Thus,
\begin{equation*}
\begin{aligned}
-  \frac{2\vec S_n\times\vec S_{n+1}}{1+\vec S_n\cdot\vec S_{n+1}} - \frac{2\vec S_n\times\vec S_{n-1}}{1+\vec S_n\cdot\vec S_{n-1}}
=P_n \begin{bmatrix}  -2\Im(\alpha_n -\alpha_{n-1})\\[1ex] -2\Re(\alpha_n -\alpha_{n-1})\\[1ex] 0\end{bmatrix},
\end{aligned}
\end{equation*}
which combined with \eqref{2:37} yields the claim.
\end{proof}

%%%%%%%%%%%%%%%%%%%%%%%%%%%%%%%%%%%%%%%%%%%%%%%%%%%%%%%%%
\section{Invariance of white noise for Ablowitz--Ladik}\label{S:AL GWP}
%%%%%%%%%%%%%%%%%%%%%%%%%%%%%%%%%%%%%%%%%%%%%%%%%%%%%%%%%

\begin{definition}
We say that a global solution $\alpha:\Z\times\R\to\C$ to the Ablowitz--Ladik system \eqref{E:AL} is a \emph{good solution} if it satisfies the following two conditions:
\begin{align}
\int_{-T}^T \sum_{n\in \Z} \langle n\rangle^{-q} |\alpha_n(t)|^{2p}\,dt <\infty \quad\text{for some $p>q>1$ and all $T>0$},\label{hyp1}\\
\sup_{|t|\leq T} \sum_{n\in \Z} e^{-c\langle n\rangle} |\alpha_n(t)|^2<\infty \quad\text{for some $c>0$ and all $T>0$}\label{hyp2}.
\end{align}
\end{definition}

\begin{remark}\label{R:GS}
If $\alpha(t)=\{\alpha_n(t)\}_{n\in\Z}$ is a good solution to \eqref{E:AL}, then so is
$$
\{e^{i\phi}\alpha_{n+m}(t+t_0)\}_{n\in \Z}
$$
for any $m\in\Z$, $\phi\in[0,2\pi)$, and $t_0\in\R$.  Indeed, one may use the same parameters $p$, $q$, and $c$ appearing in \eqref{hyp1} and \eqref{hyp2}, respectively.
\end{remark}

\begin{theorem}[Almost sure global existence and uniqueness for Ablowitz--Ladik]\label{T:AL GWP}\leavevmode\\
Fix $\beta>0$.  Then for almost every initial data $\alpha(0)=\{\alpha_n(0)\}_{n\in \Z}$ chosen according to the white noise measure
\begin{equation}\label{E:D:wn}
d\mu_{wn}^\beta=\prod_{n\in \Z}\frac{1+2\beta}{\pi} \frac {d\text{\rm Area}(\alpha_n)}{(1+|\alpha_n|^2)^{2+2\beta}}
\end{equation}
there exists a unique global good solution $\alpha:\Z\times\R\to \C$ to the Ablowitz--Ladik system \eqref{E:AL}.
\end{theorem}

\begin{proof}
We begin by constructing global solutions to \eqref{E:AL} for almost every initial data chosen according to the measure $d\mu_{wn}^\beta$.  We will do so by proving that increasingly large finite-volume solutions to the Ablowitz--Ladik system \eqref{AL free} converge to a solution to \eqref{E:AL}, uniformly on compact regions of spacetime.

Let $\alpha(0)=\{\alpha_n(0)\}_{n\in \Z}$ be chosen according to the measure $d\mu_{wn}^\beta$.  For $4\leq K\in 2^{\Z}$, let $\alpha^K:\{-K,\ldots,K\}\times\R\to\C$ denote the unique global solution to \eqref{AL free} with initial data $\alpha^K(0)=\{\alpha_n(0)\}_{|n|\leq K}$ constructed in Proposition~\ref{P:trunc}.

We will show that almost surely, the global solutions $\alpha^K$ converge uniformly on compact regions of spacetime as $K\to \infty$.  To this end, we fix $T>0$ and for each $|t|\leq T$ and $4\leq K\in 2^\Z$, we define
$$
M_K(t)= \sum_{n\in \Z} e^{-4\langle n\rangle} \bigl| \alpha_n^{2K}(t) -\alpha_n^{K}(t)\bigr|^2,
$$
with the convention that $\alpha_n^L\equiv 0$ for $|n|>L$.  Straightforward computations give
\begin{align*}
&\frac{d}{dt} M_K(t) \\
&= -2\Im \sum_{n\in \Z} e^{-4\langle n\rangle}\bigl(\overline{\alpha_n^{2K}} -\overline{\alpha_n^{K}}\bigr)\Bigl\{ (1+|\alpha_n^{2K}|^2)\Bigl[\bigl(\alpha_{n+1}^{2K} -\alpha_{n+1}^{K}\bigr) +\bigl(\alpha_{n-1}^{2K} -\alpha_{n-1}^{K}\bigr)\Bigr]\\
&\qquad\qquad\qquad\qquad\qquad+\bigl(\alpha_{n+1}^{K} +\alpha_{n-1}^{K}\bigr) \Bigl[ \overline{\alpha_n^{2K}} \bigl(\alpha_n^{2K} -\alpha_n^{K}\bigr) +\alpha_n^K \bigl(\overline{\alpha_n^{2K}} -\overline{\alpha_n^{K}}\bigr) \Bigr]\Bigr\}.
%&\quad-2\Im\sum_{K<|n|\leq 2K}e^{-4\langle n\rangle}\overline{\alpha_n^{2K}}(1+|\alpha_n^{2K}|^2)\bigl(\alpha_{n+1}^{2K}+\alpha_{n-1}^{2K}\bigr).
\end{align*}
Using Cauchy--Schwarz and the fact that $2+e^{4\langle n\rangle-4\langle n-1\rangle}+e^{4\langle n\rangle-4\langle n+1\rangle}\leq 100$ uniformly for $n\in \Z$, we get
\begin{align*}
\Bigl|\frac{d}{dt} M_K(t)\Bigr|
&\leq 100\bigl[1+\sup_n |\alpha_n^{2K}(t)|^2\bigr] M_K(t) + \bigl[6\sup_n |\alpha_n^{K}(t)|^2 +2\sup_n |\alpha_n^{2K}(t)|^2\bigr]M_K(t) \\
&\leq A(t) M_K(t),
\end{align*}
where
$$
A(t)=100 + 102 \sup_n |\alpha_n^{2K}(t)|^2+ 6 \sup_n |\alpha_n^{K}(t)|^2.
$$
Therefore, by Gronwall,
\begin{align}\label{Gronwall}
\sup_{|t|\leq T} M_K(t)\leq M_K(0) \exp\Bigl(\int_{-T}^T A(t)\, dt\Bigr).
\end{align}

To continue, we compute
\begin{align*}
\E M_K(0) =\E\sum_{n\in \Z} e^{-4\langle n\rangle} \bigl| \alpha_n^{2K}(0) -\alpha_n^{K}(0)\bigr|^2= \E \sum_{K<|n|\leq 2K} e^{-4\langle n\rangle} |\alpha_n(0)|^2 \lesssim_\beta e^{-4K}
\end{align*}
and so
\begin{align}\label{M0 large}
\Prob (M_K(0)\geq e^{-2K})\lesssim_\beta e^{-2K}.
\end{align}

Using invariance of the measure for the finite-dimensional system \eqref{AL free}, we find
\begin{align*}
\E(\sup_n |\alpha_n^L(t)|^2)=\E(\sup_n |\alpha_n^L(0)|^2)&\leq \lambda +\lambda^{-\eps} \E(\sup_n |\alpha_n^L(0)|^{2+2\eps})\\
&\leq \lambda +\lambda^{-\eps} \E \sum_{n\in \Z}|\alpha_n^L(0)|^{2+2\eps} \lesssim_\beta \lambda +\lambda^{-\eps}L,
\end{align*}
provided $\eps<2\beta$.  Optimizing in $\lambda$, we get
\begin{align*}
\E(\sup_n |\alpha_n^L(t)|^2)\lesssim_\beta L^{\frac1{1+\eps}}.
\end{align*}
Thus,
\begin{align}\label{A large}
\Prob\Bigl(\int_{-T}^T A(t)\, dt\geq K\Bigr)\leq K^{-1}\E \int_{-T}^T A(t)\, dt\lesssim_\beta K^{-1}T+ K^{-1}T K^{\frac1{1+\eps}}\lesssim_\beta TK^{-\frac\eps{1+\eps}}.
\end{align}

Combining \eqref{Gronwall} through \eqref{A large}, we obtain
$$
\sup_{|t|\leq T} M_K(t)\lesssim e^{-K}
$$
on a set $\Omega_{T,K}$ satisfying
$$
\Prob(\Omega_{T,K}^c) \lesssim_\beta \langle T\rangle (K^{-\frac\eps{1+\eps}}+ e^{-2K}),
$$
whenever $\eps<2\beta$.

Now let $\Omega_T$ be the set of initial data defined via
$$
\Omega_T=\Bigl\{\alpha(0):\, \sum_{4\leq K\in 2^\Z}\sup_{|t|\leq T}\sqrt{M_K(t)}<\infty\Bigr\} .
$$
By conservation of the Hamiltonian \eqref{H FT free}, for any $K\geq 4$ we have $\sup_{t\in \R} M_K(t)<\infty$.  Thus,
$$
\Omega_T=\bigcup_{K_0\geq 4}\Bigl\{\alpha(0):\, \sum_{K_0\leq K\in 2^\Z}\sup_{|t|\leq T}\sqrt{M_K(t)}<\infty\Bigr\}  \supseteq\bigcup_{K_0\geq 4}\bigcap_{K\geq K_0} \Omega_{T,K}.
$$
In particular, for any $\eps<2\beta$,
$$
\Prob(\Omega_T^c)\leq \sum_{K\geq K_0} \Prob(\Omega_{T,K}^c)\lesssim \langle T\rangle (K_0^{-\frac\eps{1+\eps}}+ e^{-2K_0})\to 0 \qtq{as} K_0\to \infty.
$$

Finally, let $T_n$ be a sequence of times diverging to infinity.  Then $\Omega=\bigcap \Omega_{T_n}$ is a set of full measure.  Moreover, for an initial data $\alpha(0)=\{\alpha_n(0)\}_{n\in \Z}\in\Omega$, the unique global solutions $\alpha^K:\Z\times\R\to \C$ to \eqref{AL free} with truncated initial data $\alpha^K(0) = \{\alpha_n(0)\}_{|n|\leq K}$ satisfy
$$
\sum_{4\leq K \in 2^\Z} \ \sup_{|t|\leq T,n\in\Z} e^{-2\langle n\rangle} |\alpha_n^{2K}(t) - \alpha_n^K(t)|< \infty \quad\text{for any $T>0$},
$$
which shows that $\alpha^K$ converge uniformly on compact regions of spacetime.

It follows from this that the pointwise limit $\alpha:\Z\times\R\to\C$ is a global solution to \eqref{E:AL} with initial data $\alpha(0)$.  Furthermore, for any $T>0$ this solution  satisfies
\begin{align*}
\sup_{|t|\leq T, n\in \Z} e^{-4\langle n\rangle} \bigl| \alpha_n(t)\bigr|^2<\infty,
\end{align*}
which yields \eqref{hyp2} in the definition of a good solution (with $c>4$).

Our next goal is to prove that the statistical ensemble of global solutions $\alpha$ to \eqref{E:AL} that we constructed above satisfies
\begin{align}\label{Ehyp1}
\E\int_{-T}^T \sum_{n\in \Z} \langle n\rangle^{-q} |\alpha_n(t)|^{2p}\,dt <\infty
\end{align}
for any $1<q<p<1+2\beta$ and any $T>0$.  In this way, we see that \eqref{E:AL} admits a global good solution for a full measure set of initial data.

Fix $T>0$ and $4\leq K\in 2^\Z$.  By invariance of the measure for the finite-dimensional system \eqref{AL free}, we obtain
\begin{align*}
\E\int_{-T}^T\sum_{n\in \Z} \, \langle n\rangle ^{-q} |\alpha_n^K(t)|^{2p}\,dt
&=\int_{-T}^T\E \sum_{n\in\Z} \,\langle n\rangle ^{-q} |\alpha_n^K(0)|^{2p}\, dt \lesssim_\beta T,
\end{align*}
provided merely $q>1$ and $p<1+2\beta$.  As $\alpha^K$ converge uniformly on compact regions of spacetime to $\alpha$, Fatou's Lemma implies \eqref{Ehyp1}.

Finally, it remains to prove uniqueness in the class of good solutions.  Let $\alpha(t)$ and $\beta(t)$ be two good solutions to \eqref{E:AL} with initial data $\alpha(0)=\beta(0)$.  Assume, towards a contradiction, that the two solutions $\alpha$ and $\beta$ are not equal.  Then, translating in space (cf. Remark~\ref{R:GS}) and reversing time if necessary, we may find $T>0$ so that
\begin{align}\label{different}
\alpha_0(T)\neq \beta_0(T).
\end{align}

As $\alpha$ and $\beta$ verify \eqref{hyp1} and \eqref{hyp2}, there exist $\sigma\in(0,1)$ and positive constants $A_T$, $c$, and $B_T$ such that
\begin{gather}
\int_{-T}^T \; \sup_{|n| \leq 2N} \Bigl[ 1+ |\alpha_n(t)|^{2} + |\beta_n(t)|^{2}\Bigr]\,dt\leq A_T N^{\sigma}\quad\text{uniformly for $N\geq 1$},\label{hyp1'}\\
\sup_{|t|\leq T} \sum_{n\in \Z} e^{-c| n|} \Bigl[ 1+ |\alpha_n(t)|^2 + |\beta_n(t)|^2\Bigr]\leq B_T\label{hyp2'}.
\end{gather}
Indeed, in terms of the parameters appearing in \eqref{hyp1}, we may take
$$
\sigma=\max\bigl\{\tfrac{q_\alpha}{p_\alpha},\tfrac{q_\beta}{p_\beta}\bigr\} \qtq{and} c = \max\{c_\alpha,c_\beta\}.
$$

To continue, for $t\in[-T,T]$ we define
$$
M(t) = \sum_{n\in Z} e^{-3c |n|} \bigl|\alpha_n(t) -\beta_n(t)\bigr|^2.
$$
A straightforward computation yields
\begin{align*}
\bigl|\tfrac{d M}{dt} \bigr| &\leq  \sum_{n\in \Z} e^{-3c|n|} (1+ |\alpha_n|^2)\Bigl[ 2|\alpha_n -\beta_n|^2+ |\alpha_{n+1} -\beta_{n+1}|^2 +|\alpha_{n-1} -\beta_{n-1}|^2\Bigr]\\
&\quad + \sum_{n\in \Z} e^{-3c|n|}\; 2 |\alpha_n -\beta_n|^2 \bigl(|\alpha_n| + |\beta_n|\bigr)\bigl( |\beta_{n+1}| + |\beta_{n-1}|\bigr)\\
&\leq C e^{3c} \sup_{|n|\leq 2N} (1+ |\alpha_n(t)|^2 +|\beta_n(t)|^2) \; M(t) \\
&\quad+ C e^{3c} \sum_{|n|\geq N}e^{-3c|n|} \bigl(1+|\alpha_n(t)|^2 +|\beta_n(t)|^2\bigr)^2,
\end{align*}
for some absolute constant $C$ and  $N\geq 2$.  Now employing \eqref{hyp2'} we obtain
$$
\bigl|\tfrac{d M}{dt} \bigr| \leq C e^{3c} \Bigl\{ \sup_{|n|\leq 2N} (1+ |\alpha_n(t)|^2 +|\beta_n(t)|^2) \, M(t) +  e^{-cN} B_T^2 \Bigr\}
$$
uniformly for $t\in[-T,T]$ and $N\geq2$. By Gronwall and \eqref{hyp1'}, this implies
\begin{align*}
M(T) & \leq C e^{3c} T B_T ^2 \exp\Bigl\{-cN + C e^{3c} A_T N^\sigma \Bigr\}.
\end{align*}
This contradicts \eqref{different}, since the right-hand side above converges to zero as $N\to \infty$, thereby completing the proof of uniqueness.
\end{proof}

\begin{theorem}[Invariance of white noise for Ablowitz--Ladik]\label{T:wn invariance}
Fix $\beta>0$.  Then the white noise measure $d\mu_{wn}^\beta$ is left invariant by the flow of the Ablowitz--Ladik system \eqref{E:AL}.
\end{theorem}

\begin{proof}
Let $\alpha(0)=\{\alpha_n(0)\}_{n\in \Z}$ belong to the full-measure set of initial data for which Theorem~\ref{T:AL GWP} guarantees the existence of a unique global good solution to \eqref{E:AL} and let $\alpha:\Z\times\R\to\C$ denote this solution.  To prove invariance of the white noise measure, it suffices to show that
\begin{align*}
\int F( \alpha(t)) \, d\mu_{wn}^\beta (\{\alpha_n(0)\})=\int F( \alpha(0)) \, d\mu_{wn}^\beta (\{\alpha_n(0)\})
\end{align*}
for all $t\in \R$ and all bounded continuous functions $F$ depending on only finitely many coordinates.

To proceed, we fix such an $F$ and choose $N$ large enough so that $F$ is determined by $\alpha_{-N},\ldots,\alpha_N$. For $K\geq N$, let $\alpha^K$ denote the unique global solution to \eqref{AL free} with data $\alpha^K(0)=\{\alpha_n(0)\}_{|n|\leq K}$; see Proposition~\ref{P:trunc}.  This proposition also shows that the measure
$$
d\mu_{wn}^{\beta,K}(\{\alpha_n(0)\})= \prod_{-K\leq n\leq K}\frac{1+2\beta}{\pi} \frac {d\text{Area}(\alpha_n)}{(1+|\alpha_n|^2)^{2+2\beta}}
$$
is left invariant by this flow.  Thus for any $t\in \R$,
\begin{align*}
\int F( \alpha(0)) \, d\mu_{wn}^\beta (\{\alpha_n(0)\})
&=\int F( \alpha_{-N}(0), \ldots, \alpha_N(0)) \, d\mu_{wn}^\beta (\{\alpha_n(0)\})\\
&=\int F( \alpha_{-N}(0), \ldots, \alpha_N(0)) \, d\mu_{wn}^{\beta,K}(\{\alpha_n(0)\})\\
&=\int F( \alpha_{-N}^K(t), \ldots, \alpha_N^K(t)) \, d\mu_{wn}^{\beta,K}(\{\alpha_n(0)\})\\
&=\int F( \alpha_{-N}^K(t), \ldots, \alpha_N^K(t)) \, d\mu_{wn}^\beta (\{\alpha_n(0)\}).
\end{align*}
As $\alpha^K$ converges to $\alpha$ uniformly on compact regions of spacetime as $K\to \infty$, so
\begin{align*}
\int F( \alpha_{-N}^K(t), \ldots \alpha_N^K(t)) \, d\mu_{wn}^\beta (\{\alpha_n(0)\})\to\int F( \alpha_{-N}(t), \ldots \alpha_N(t)) \, d\mu_{wn}^\beta (\{\alpha_n(0)\})
\end{align*}
as $K\to \infty$.  This completes the proof of the theorem.
\end{proof}

%%%%%%%%%%%%%%%%%%%%%%%%%%%%%%%%%%%%%%%%%%%%%%%%%%%%%%%%%
\section{Invariance of the Gibbs measure for the spin model}\label{S:5}
%%%%%%%%%%%%%%%%%%%%%%%%%%%%%%%%%%%%%%%%%%%%%%%%%%%%%%%%%

In this section we prove almost sure global existence and uniqueness for the spin chain model \eqref{E:S} with initial data distributed according to the Gibbs measure.  Moreover, we show that the flow of \eqref{E:S} leaves the Gibbs measure invariant.

Our first task is to make sense of the Gibbs measure for \eqref{E:S}. We say that a measure with expectation $\E_\beta$ is \emph{a Gibbs measure} at inverse-temperature $\beta$ for \eqref{E:S} if it satisfies the DLR condition.  This condition takes its name from the work of Dobrushin, \cite{Dob}, and Lanford--Ruelle, \cite{LR}.  In the setting of our model, it says the following: for any bounded and continuous function $f$ and any integers $a\leq b$,
\begin{align}\label{DLR}
\E_\beta\bigl\{ & f\bigl(\vec S_{a}, \cdots, \vec S_{b}\bigl)\bigl| \,\vec S_{a-1}, \vec S_{b+1}\bigr\}\\
&= \frac1{Z_{ab}}\int_{\mathbb{S}^2}\!\cdots\!\int_{\mathbb{S}^2} \!f\bigl( s_{a}, \cdots, s_{b}\bigl)p(\vec S_{a-1}, s_a) p(s_b, \vec S_{b+1}) \prod_{k=a}^{b-1}  p(s_k,s_{k+1}) ds_{a} \cdots  ds_b,\notag
\end{align}
where
\begin{equation}\label{little p}
p(s,\sigma)= \tfrac{1+2\beta}{4\pi} \exp\bigl\{2\beta\log\bigl( 1-\tfrac14|s-\sigma|^2\bigr)\bigr \} = \tfrac{1+2\beta}{4\pi} \bigl(\tfrac{1+s\cdot \sigma}2\bigr)^{2\beta},
\end{equation}
as dictated by \eqref{HS}.  The numerical factor $\tfrac{1+2\beta}{4\pi}$ is included here for later convenience.  It is inconsequential in \eqref{DLR}, because it is canceled by the normalization constant
$$
Z_{a,b} = \int_{\mathbb{S}^2}\!\cdots\!\int_{\mathbb{S}^2} p(\vec S_{a-1}, s_a) p(s_b,\vec S_{b+1})\prod_{k=a}^{b-1} p(s_k,s_{k+1})\,ds_{a}\cdots ds_{b}.
$$
Here and below, integration over the sphere is performed with respect to area measure; hence $\int_{\mathbb S^2}\,ds = 4\pi$.  This is dictated by the symplectic structure underlying Definition~\ref{D:pb}.

\begin{proposition}[Existence and uniqueness of the Gibbs measure]\label{P:Gibbs}
The spin chain model \eqref{E:S} admits a unique Gibbs measure at inverse-temperature $\beta>0$.  Moreover, for any integers  $n\leq m$,
\begin{align}\label{uncond exp}
\E_\beta \Bigl\{&f\bigl( \vec S_n, \ldots, \vec S_{m}\bigl)\Bigr\}= \int_{\mathbb{S}^2} \cdots\int_{\mathbb{S}^2} f\bigl( s_n, \ldots, s_{m}\bigl) \prod_{k=n}^{m-1} p(s_k,s_{k+1})\, ds_n \cdots ds_m,
\end{align}
using the notation \eqref{little p}.  We denote this Gibbs measure by $d\mu_{Gibbs}^\beta$.
\end{proposition}

\begin{remark}
The law \eqref{uncond exp} shows that the random variables $\{\vec S_n\}$ can also be interpreted as the stationary Markov chain associated to the transition probabilities
$$
\E_\beta\Bigl\{ f\bigl(\vec S_{n+1}\bigr)\bigl|\vec S_n \Bigr\} =\int_{\mathbb{S}^2} f\bigl(s\bigr) p(s,\vec S_n)\, ds.
$$
\end{remark}

\begin{proof}
The formula \eqref{uncond exp} gives a consistent family of marginals.  Thus, by Kolmogorov's extension theorem there exists a unique probability measure with these marginals.  It is easy to verify directly from \eqref{uncond exp} that this probability measure satisfies the DLR condition \eqref{DLR}.  It thus remains to verify that any law $\E_\beta$ satisfying the DLR condition \eqref{DLR} has marginals given by \eqref{uncond exp}.

To continue, we define inductively the kernels $p_k:\mathbb{S}^2\times \mathbb{S}^2\to \R$ via
$$
p_1(s,\sigma)=  p(s,\sigma) \qtq{and} p_{k+1}(s,\sigma)= \int_{\mathbb{S}^2} p_k(s,v) p(v,\sigma) \, dv.
$$
With this notation, \eqref{DLR} implies that for any integers $a<n\leq m<b$,
\begin{align*}
\E_\beta \Bigl\{f\bigl( \vec S_n, \cdots, \vec S_{m}\bigl)\Bigr\}
&=\E_\beta\Bigl\{\E_\beta\bigl\{f\bigl( \vec S_n, \cdots, \vec S_{m}\bigl)\bigl| \, \vec S_a, \vec S_b\bigr\}\Bigr\}\\
&=\E_\beta\Biggl\{\int_{\mathbb{S}^2}\cdots\int_{\mathbb{S}^2} \frac{p_{n-a}(\vec S_a,s_n)p_{b-m}(s_{m},\vec S_b)}{p_{b-a}(\vec S_a,\vec S_b)} f\bigl( s_n, \cdots, s_m\bigl)\\
&\qquad \qquad \qquad \qquad \qquad \qquad\times \prod_{k=n}^{m-1}p_1(s_k,s_{k+1})\,ds_n \cdots ds_m \Biggr\}.
\end{align*}
To obtain \eqref{uncond exp}, it thus suffices to show that
\begin{align}\label{1137}
p_k(s, \sigma) \to \tfrac1{4\pi}\quad\text{uniformly as $k\to \infty$}.
\end{align}

Let $P$ denote the operator with kernel $p_1$; this operator is compact, self-adjoint, and positivity-improving; moreover, the constant functions are eigenvectors with eigenvalue $1$.  Therefore, by the Perron--Frobenius theorem, $P^k$ converges in operator norm to projection onto constant functions as $k\to \infty$.  Writing
$$
p_{k+2}(s,\sigma) = \langle p(s,\cdot), P^k p(\cdot, \sigma)\rangle_{L^2(\mathbb{S}^2)},
$$
this immediately implies \eqref{1137} and so completes the proof of the proposition.
\end{proof}

Now that we have established existence and uniqueness of the Gibbs measure for the spin chain model \eqref{E:S} at inverse temperature $\beta>0$, we wish to prove almost sure global existence and uniqueness of solutions to \eqref{E:S} for data distributed according to this measure.  We will work with the following notion of solution:

\begin{definition}
We say that a global solution $\vec S:\R\times\Z\to \mathbb{S}^2$ to the spin chain model \eqref{E:S} is a \emph{good solution} if it satisfies the following:
\begin{align}
\int_{-T}^T \sum_{n\in \Z} \frac{\langle n\rangle^{-q}}{\bigl[1+\vec S_n(t) \cdot \vec S_{n+1}(t)\bigr]^p}\,dt &<\infty \quad\text{for some $p>q>1$ and all $T>0$,}\label{hyp1S}\\
\sup_{|t|\leq T} \sum_{n\in \Z}  \frac{e^{-c\langle n\rangle}}{1+\vec S_n(t) \cdot \vec S_{n+1}(t)}&<\infty \quad\text{for some $c>0$ and all $T>0$}\label{hyp2S}.
\end{align}
\end{definition}

Note that the property of being a good solution is invariant under rigid rotations (the natural gauge transformations), as well as space and time translations.  In view of the denominators in \eqref{E:S}, it is necessary to avoid consecutive spins being anti-parallel.  The above restriction is a more quantitative version of this that allows us to prove uniqueness and is connected to our notion of good solution to \eqref{E:AL} via the discrete Hasimoto transform.  We do not know if uniqueness holds for completely general classical solutions to \eqref{E:S}.

\begin{proposition}[Uniqueness of good solutions]\label{P:uniq}
Let $\vec S(t)$ and $\vec U(t)$ be global good solutions to \eqref{E:S} with initial data $\vec S(0) = \vec U(0)$.  Then $\vec S(t)=\vec U(t)$ for all $t\in\R$.
\end{proposition}

\begin{proof}
Fix $T>0$.  As $\vec S$ and $\vec U$ verify \eqref{hyp1S} and \eqref{hyp2S}, there exist $\sigma\in(0,1)$, $c>0$, and positive constants $A_T$ and $B_T$ such that
\begin{align}
\int_{-T}^T \sup_{|n|\leq 2N} \biggl[1+ \frac{1}{1+\vec S_n(t) \cdot \vec S_{n+1}(t)}+ \frac{1}{1+\vec U_n(t) \cdot \vec U_{n+1}(t)}\biggr] \,dt\leq A_T N^\sigma,\label{hyp1S'}
\end{align}
uniformly for $N\geq 2$ and
\begin{align}
\sup_{|t|\leq T} \sum_{n\in \Z} e^{-c| n|} \biggl[1+ \frac{1}{1+\vec S_n(t) \cdot \vec S_{n+1}(t)}+ \frac{1}{1+\vec U_n(t) \cdot \vec U_{n+1}(t)}\biggr]\leq B_T\label{hyp2S'}.
\end{align}

To continue, for $t\in[-T,T]$ we define
$$
M(t) = \sum_{n\in Z} e^{-2c|n|} \bigl|\vec S_n(t) -\vec U_n(t)\bigr|^2,
$$
where $C>0$ denotes a large constant to be chosen later.  A straightforward computation yields

\begin{align*}
\frac{d M}{dt} =-4\sum_{n\in\Z}e^{-2c|n|}  (\vec S_n-\vec U_n) \cdot & \biggl\{ \frac{\vec S_n+\vec S_{n+1}}{|\vec S_n+\vec S_{n+1}|^2}\times \vec S_{n+1}-\frac{\vec U_n+\vec U_{n+1}}{|\vec U_n+\vec U_{n+1}|^2}\times \vec U_{n+1}\\
&+ \frac{\vec S_n+\vec S_{n-1}}{|\vec S_n+\vec S_{n-1}|^2}\times \vec S_{n-1}-\frac{\vec U_n+\vec U_{n-1}}{|\vec U_n+\vec U_{n-1}|^2}\times \vec U_{n-1} \biggr\}.
\end{align*}
Using $|\vec a\times \vec b- \vec c\times \vec d|\leq |\vec a-\vec c| |\vec b| + |\vec c| |\vec b- \vec d|$ followed by the arithmetic--geometric mean inequality, we get
\begin{align*}
\Bigl|\frac{dM}{dt} \Bigr|&\leq 4\sum_{n\in\Z}e^{-2c|n|} |\vec S_n-\vec U_n|  \biggl\{ \frac{|\vec S_n-\vec U_n|+ | \vec S_{n+1}-\vec U_{n+1}|}{|\vec S_n+\vec S_{n+1}||\vec U_n+\vec U_{n+1}|} + \frac{|\vec S_{n+1}-\vec U_{n+1}|}{|\vec U_n+\vec U_{n+1}|}\\
&\qquad\qquad\qquad\qquad\qquad\quad+ \frac{|\vec S_n-\vec U_n|+| \vec S_{n-1}-\vec U_{n-1}|}{|\vec S_n+\vec S_{n-1}||\vec U_n+\vec U_{n-1}|} + \frac{|\vec S_{n-1}-\vec U_{n-1}|}{|\vec U_n+\vec U_{n-1}|}\biggr\}\\
&\leq C e^{2c} \sup_{|n|\leq 2N} \biggl\{ 1 + \frac{1}{1+\vec S_n \cdot \vec S_{n+1}}+ \frac{1}{1+\vec U_n \cdot \vec U_{n+1}} \biggr\} M(t)\\
& \quad+ C e^{2c}  \sum_{|n|\geq N} e^{-2c|n|} \biggl\{ 1 + \frac{1}{1+\vec S_n \cdot \vec S_{n+1}}+ \frac{1}{1+\vec U_n \cdot \vec U_{n+1}} \biggr\}
\end{align*}
for some absolute constant $C$ and any $N\geq2$.  As $M(0)=0$ by assumption, combining Gronwall with \eqref{hyp1S'} and \eqref{hyp2S'} yields
\begin{align*}
\sup_{|t|\leq T} |M(t)| & \leq  C e^{2c} T B_T \exp\bigl\{-cN + C e^{2c} A_T N^\sigma \bigr\} \to 0 \qtq{as} N\to\infty.
\end{align*}
Therefore, $M(t)=0$ for all $|t|\leq T$.  As $T>0$ was arbitrary, this shows that $S(t)=U(t)$ for all $t\in \R$.
\end{proof}

We are now ready to tackle Theorem~\ref{T:S GWP}, whose proof will occupy the remainder of this section.

\begin{proof}[Proof of Theorem~\ref{T:S GWP}]

We first address the existence of global good solutions to \eqref{E:S}, for which we will rely on the results of Sections~\ref{S:AL to S} and \ref{S:AL GWP}.  Specifically, Theorem~\ref{T:AL GWP} guarantees the existence of a full measure set of initial data distributed according to the white noise measure $d\mu_{wn}^\beta$ for which there exist unique global good solutions to \eqref{E:AL}. Let $\alpha(0)$ belong to this full measure set of initial data and let $\alpha:\Z\times\R\to \C$ denote the unique global good solution to \eqref{E:AL} with initial data $\alpha(0)$.  Let $\mathcal O\in \SO(3)$ be an independent random variable distributed according to Haar measure. (This plays the role of a random choice of gauge.)  For $n\in \Z$, we define $Q_n(t)$ and $P_n(t)$ as in \eqref{Qn} through \eqref{Pn}.  By Theorem~\ref{T:AL to S},  $\vec S(t) = \{\vec S_n(t)\}_{n\in\Z}$ defined as in \eqref{S def} is a global solution to \eqref{E:S}.  Moreover, since $\alpha$ verifies \eqref{hyp1} and \eqref{hyp2}, it is easy to check that $\vec S$ verifies \eqref{hyp1S} and \eqref{hyp2S}, and so it is a global good solution to \eqref{E:S}.  Proposition~\ref{P:uniq} shows that this solution is uniquely determined by the initial data.  This is important since (due to gauge invariance) each initial configuration $\vec S(0)$ results from continuum many choices of $\alpha(0)$ and $\mathcal O$.

Next, we have to verify that the initial data $\vec S(0)$ for the solution to \eqref{E:S} constructed above is indeed distributed according to the Gibbs measure $d\mu_{Gibbs}^\beta$.  This is the scope of the next proposition.  In fact, together with Proposition~\ref{P:uniq}, our next result also proves that the Gibbs measure $d\mu_{Gibbs}^\beta$ is left invariant by the flow of \eqref{E:S}, thus completing the proof of Theorem~\ref{T:S GWP}.

\begin{proposition}\label{P:same}
For any $t\in \R$, the sequence $\vec S(t) = \{\vec S_n(t)\}_{n\in\Z}$ is distributed according to the Gibbs measure $d\mu_{Gibbs}^\beta$.
\end{proposition}

\begin{proof}
The proof proceeds in two steps: First we verify the invariance of the joint law $ d\Haar\, d\mu_{wn}^\beta$ under the flow given by \eqref{E:AL} and \eqref{P0}.  Then we prove that the measure on the spins induced by $d\Haar\, d\mu_{wn}^\beta$ agrees with $d\mu_{Gibbs}^\beta$.

$\bf{Step \,1.}$ To verify invariance of the joint law $ d\Haar\, d\mu_{wn}^\beta$ under the flow given by \eqref{E:AL} and \eqref{P0}, it suffices to show that for any $N\geq 0$ and any bounded continuous function $F:\SO(3)\times\R^{2N+1}\to \R$ we have
\begin{equation}\label{E:step1}
\begin{aligned}
&\iint F\bigl(P_0(t), \alpha_{-N}(t), \ldots, \alpha_N(t)\bigr)\,d\Haar (P_0(0))\, d\mu_{wn}^\beta(\{\alpha(0)\})\\
={}& \iint F\bigl(P_0(0), \alpha_{-N}(0), \ldots, \alpha_N(0)\bigr)\, d\Haar (P_0(0))\, d\mu_{wn}^\beta(\{\alpha(0)\})
\end{aligned}
\end{equation}
for all $t\in\R$.

To this end, let $\mathcal A$ denote the $\sigma$-algebra generated by the random variables $\alpha_n(0)$.  For a full measure set of initial data, there exists a unique global good solution $\alpha(t)$ to \eqref{E:AL}.  This shows that $\alpha(t)$ is $\mathcal A$-measurable for all $t\in \R$.  Moreover, defining $A_0(t)$ via \eqref{An} and then $\Phi(t)$ by
$$
\frac{d}{dt}\Phi(t) = \Phi(t) A_0(t) \qtq{with} \Phi(0) = \Id,
$$
we see that $\Phi(t)$ is also $\mathcal A$-measurable.  Note that $P_0(0)$ is independent of $\mathcal A$.

Thus, by right-invariance of the Haar measure followed by invariance of the white noise measure under the flow of \eqref{E:AL}, we obtain
\begin{align*}
\text{LHS\eqref{E:step1}} &= \E_\beta\Bigl\{\E_\beta\bigl\{F\bigl(P_0(0)\Phi(t),\alpha_{-N}(t), \ldots, \alpha_N(t)\bigr) \big| \mathcal A\bigr\}\Bigr\}\\
&= \iint F\bigl(\mathcal O, \alpha_{-N}(t), \ldots, \alpha_N(t)\bigr)\,d\Haar (\mathcal O)\, d\mu_{wn}^\beta(\{\alpha(0)\})\\
&= \iint F\bigl(\mathcal O, \alpha_{-N}(0), \ldots, \alpha_N(0)\bigr)\, d\Haar (\mathcal O)\, d\mu_{wn}^\beta(\{\alpha(0)\}) = \text{RHS\eqref{E:step1}}.
\end{align*}
This proves invariance of the joint law $ d\Haar\, d\mu_{wn}^\beta$.

These arguments also yield the law of a single spin:  In view of \eqref{Pn}, for any $n\in \Z$ and $t\in\R$, there is an $\mathcal A$-measurable matrix $\Phi_n(t)\in SO(3)$ so that
$$
\vec S_n(t) = P_0(0) \Phi_n(t) \vec e_3; \qtq{indeed,} \Phi_n(t) =\begin{cases} \Phi(t) Q_0(t) \cdots Q_{n-1}(t) &: n\geq 0 \\ \Phi(t) Q_{-1}(t)^T \cdots Q_{n}(t)^T  & : n\leq 0.\end{cases}
$$
As  $P_0(0)$ is Haar distributed and independent of $\mathcal A$,
\begin{align}\label{2:45}
\E_\beta\bigl\{ g\bigl(\vec S_n(t) \bigr)\bigr\} = \E_\beta\Bigl\{ \E_\beta\bigl\{ g\bigl(P_0(0) \Phi_n(t) \vec e_3 \bigr) \big| \mathcal A\bigr\} \Bigr\} = \tfrac1{4\pi}\int_{\mathbb S^2} g(s)\,ds.
\end{align}

$\bf{Step \,2.}$ To verify that the measure induced by the joint law $ d\Haar\, d\mu_{wn}^\beta$ on the spins $\{ \vec S_n(t)\}_{n\in \Z}$ agrees with the Gibbs measure $d\mu_{Gibbs}^\beta$, it suffices to verify that the induced measure gives the same marginals as \eqref{uncond exp}.

To this end, fix $t\in \R$.  For $k\in\Z$, we let $\mathcal A_k$ denote the $\sigma$-algebra generated by the random variables $\{P_n(t)\}_{n\leq k}$, or equivalently, by $\{P_{k}(t), \{\alpha_n(t)\}_{n\leq k-1}\}$.  Note that $\vec S_l(t)$ is $\mathcal A_k$ measurable if and only if $l\leq k$.

The key observation is the following:

\begin{lemma}\label{L:cond}
For any bounded and continuous function $f$ and any integers $n\leq m$,
\begin{equation}\label{E:L:cond}
\begin{aligned}
\E_\beta\Bigl\{ &f\bigl(\vec S_n(t), \ldots, \vec S_m(t)\bigr) \Big| \mathcal{A}_{m-1} \Bigr\}\\
&= \int_{\mathbb{S}^2} f\bigl(\vec S_n(t), \ldots, \vec S_{m-1}(t), s_{m} \bigr)  p\bigl(\vec S_{m-1}(t),s_{m}\bigr)\, ds_m.
\end{aligned}
\end{equation}
\end{lemma}

\begin{proof}
We use the notation of Section~\ref{S:AL to S}. As $\alpha_{m-1}(t)$ is independent of $\mathcal A_{m-1}$,
\begin{align*}
\E_\beta\bigl\{ g\bigl(Q_{m-1}(t)\vec e_3\bigl)\big|\mathcal A_{m-1}\bigr\}
=\int_0^{2\pi} \int_0^\infty g\Bigl( \frac1{1+r^2}\left[\begin{smallmatrix} 2r\cos(\theta) \\ 2r\sin(\theta) \\ 1-r^2\end{smallmatrix}\right] \Bigr) \frac{(1+2\beta) r\,dr\, d\theta}{\pi (1+r^2)^{2+2\beta}} 
\end{align*}
for any bounded and continuous function $g$. Here we used polar coordinates in the form $\alpha_{m-1}(t) = r e^{-i\theta}$.  Changing variables via $\cos(\phi) = \frac{1-r^2}{1+r^2}$ with $\phi\in[0,\pi)$ yields
\begin{align*}
\E_\beta\bigl\{ g\bigl(Q_{m-1}(t)\vec e_3\bigl)&\big|\mathcal A_{m-1}\bigr\} \\
&=\int_0^{2\pi} \!\int_0^{\pi} \!g\Bigl(\left[\begin{smallmatrix} \sin(\phi) \cos(\theta) \\ \sin(\phi) \sin(\theta) \\ \cos(\phi)\end{smallmatrix}\right] \Bigr) \tfrac{1+2\beta}{4\pi} \Bigl[\frac{1+\cos(\phi)}2\Bigr]^{2\beta} \sin(\phi) \,d\phi\, d\theta\\
&=\int_{\mathbb{S}^2} g (s) \tfrac{1+2\beta}{4\pi} \Bigl[\frac{1+ s\cdot \vec e_3}2\Bigr]^{2\beta} \,ds,
\end{align*}
where we used spherical coordinates to obtain the last equality.  Consequently,
\begin{align*}
\text{LHS\eqref{E:L:cond}} &= \E_\beta\Bigl\{ f\bigl(\vec S_n(t), \ldots,\vec S_{m-1}(t),P_{m-1}(t)Q_{m-1}(t)\vec e_3\bigr) \bigl| \mathcal A_{m-1}\Bigr\} \\
&=\int_{\mathbb{S}^2} f\bigl(\vec S_n(t), \ldots,\vec S_{m-1}(t),s\bigr) \tfrac{1+2\beta}{4\pi} \Bigl[\tfrac{1+s\cdot \vec S_{m-1}(t)}2\Bigr]^{2\beta} \,ds = \text{RHS\eqref{E:L:cond}}.
\end{align*}
This completes the proof of the lemma.
\end{proof}

Applying Lemma~\ref{L:cond} inductively and then \eqref{2:45}, we obtain
\begin{align*}
\E_\beta\Bigl\{& f\bigl(\vec S_n(t), \ldots \vec S_{m}(t)\bigr)\Bigr\}\\
&= \E_\beta\Bigl\{\int_{\mathbb{S}^2} \cdots \int_{\mathbb{S}^2} f\bigl( s_n, \cdots, s_{m}\bigl) p\bigl(\vec S_{n-1}(t), s_n\bigr)\prod_{k=n}^{m-1} p(s_k,s_{k+1})\,ds_n \cdots ds_m \Bigr\}\\
&=\int_{\mathbb{S}^2}  \cdots \int_{\mathbb{S}^2}  f\bigl( s_n, \cdots, s_{m}\bigl)\prod_{k=n}^{m-1} p(s_k,s_{k+1})\,ds_n \cdots ds_m,
\end{align*}
which agrees with the Gibbs marginals appearing in \eqref{uncond exp}.
\end{proof}

To recapitulate, Proposition~\ref{P:same} shows that there exists a full measure set of initial data for which one can construct global good solutions to \eqref{E:S}.  Proposition~\ref{P:uniq} then guarantees the uniqueness of these global good solutions for a full measure set of initial data.  Finally, Proposition~\ref{P:same} proves that the Gibbs measure $d\mu_{Gibbs}^\beta$ is left invariant by the flow of \eqref{E:S}, thus completing the proof of Theorem~\ref{T:S GWP}.
\end{proof}

%%%%%%%%%%%


\begin{thebibliography}{100}

\bibitem{AL} M. J. Ablowitz and J. F. Ladik,
Nonlinear differential-difference equations.
J. Mathematical Phys. \textbf{16} (1975), 598--603. %MR0377223

%\bibitem{BKPtoda} D. Bambusi, T. Kappeler, and T. Paul, From Toda to KdV.  Nonlinearity \textbf{28} (2015), no. 7, 2461--2496. %MR3366652

\bibitem{BanVega} V. Banica and L. Vega,
The initial value problem for the binormal flow with rough data.
Ann. Sci. \'Ec. Norm. Sup\'er. (4) \textbf{48} (2015), no. 6, 1423--1455.  %MR3429472

\bibitem{Bishop} R. L. Bishop, There is more than one way to frame a curve.
Amer. Math. Monthly \textbf{82} (1975), 246--251. %MR0370377

\bibitem{Bourg:book} J. Bourgain, Global solutions of nonlinear Schr\"odinger equations.
American Mathematical Society Colloquium Publications, \textbf{46}. American Mathematical Society, Providence, RI, 1999. %MR1691575

\bibitem{Bourg:IMNLS} J. Bourgain, Invariant measures for NLS in infinite volume.
Comm. Math. Phys. \textbf{210} (2000), no. 3, 605--620. %MR1777342

\bibitem{CarlesKappeler} R. Carles and T. Kappeler,
Norm-inflation with infinite loss of regularity for periodic NLS equations in negative Sobolev spaces.
Bull. Soc. Math. France \textbf{145} (2017), no. 4, 623--642. %MR3770969

\bibitem{CSU} N.-H. Chang, J. Shatah, and K. Uhlenbeck,
Schr\"odinger maps.
Comm. Pure Appl. Math. \textbf{53} (2000), no. 5, 590--602.  %MR1737504

\bibitem{Christ} M. Christ,
Power series solution of a nonlinear Schr\"odinger equation. \emph{Mathematical aspects of nonlinear dispersive equations}, 131--155,
Ann. of Math. Stud., \textbf{163}, Princeton Univ. Press, Princeton, NJ, 2007. %MR2333210

\bibitem{CollianderOh} J. Colliander and T. Oh,
Almost sure well-posedness of the cubic nonlinear Schr\"odinger equation below $L^2(\mathbb{T})$.
Duke Math. J. \textbf{161} (2012), no. 3, 367--414.

%\bibitem{Date} E. Date, M. Jimbo, and T. Miwa, Method for Generating Discrete Soliton Equations. I.
%J. Phys. Soc. Jpn. \textbf{51} (1982), no. 12, 4116--4124.

\bibitem{Ding} W. Ding,
On the Schr\"odinger flows.
\emph{Proceedings of the International Congress of Mathematicians (Beijing, 2002), Vol. II}, 283--291, Higher Ed. Press, Beijing, 2002.

\bibitem{Dob} P. L. Dobruschin,
The description of a random field by means of conditional probabilities and conditions of its regularity.
Theory Probab. Appl. \textbf{13} (1968), no. 2, 197--224.

\bibitem{DS} A. Doliwa and P. M. Santini,
Integrable dynamics of a discrete curve and the Ablowitz-Ladik hierarchy. J. Math. Phys. \textbf{36} (1995), no. 3, 1259--1273. %MR1317439

\bibitem{ErcLozano}
N. M. Ercolani and G. Lozano,
A bi-Hamiltonian structure for the integrable, discrete non-linear Schr\"odinger system.
Phys. D \textbf{218} (2006), no. 2, 105--121. %MR2238744

\bibitem{FT:book} L. D. Faddeev and L. Takhtajan,
\emph{Hamiltonian methods in the theory of solitons.} Classics in Mathematics. Springer, Berlin, 2007. %MR2348643

\bibitem{Semiclassical} J. Fr\"ohlich, A. Knowles, and E. Lenzmann,
Semi-classical dynamics in quantum spin systems.
Lett Math Phys \textbf{82} (2007), no. 2--3, 275--296.

\bibitem{Gilbert} T. L. Gilbert,
A phenomenological theory of damping in ferromagnetic materials.
IEEE Transactions on Magnetics \textbf{40} (2004), no. 6, 3443--3449.

\bibitem{GrunHerr}
A. Gr\"unrock and S. Herr,
Low regularity local well-posedness of the derivative nonlinear Schr\"odinger equation with periodic initial data.
SIAM J. Math. Anal. \textbf{39} (2008), no. 6, 1890--1920. %MR2390318

\bibitem{GuoOh} Z. Guo and T. Oh,
Non-Existence of Solutions for the Periodic Cubic NLS below $L^2$.
Int. Math. Res. Not. IMRN 2018, no. 6, 1656--1729. %MR3801473

\bibitem{Haldane} F. D. M. Haldane, Excitation spectrum of a generalised Heisenberg ferromagnetic spin chain with arbitrary spin.
J. Phys. C: Solid State Physics \textbf{15} (1982), no. 36, L1309--L1312.

\bibitem{Hasimoto} H. Hasimoto, A soliton on a vortex filament. J. Fluid Mech. \textbf{51} (1972), no. 3, 477--485. %MR3363420

\bibitem{HongYang}
Y. Hong and C. Yang,
Strong convergence for discrete nonlinear Schr\"odinger equations in the continuum limit.
Preprint \texttt{arXiv:1806.07542}.

\bibitem{Ishimori} Y. Ishimori, An integrable classical spin chain.
J. Phys. Soc. Jpn. \textbf{51} (1982), no. 11, 3417--3418.

\bibitem{IK} A. G. Izergin and V. E. Korepin,
A lattice model connected with a nonlinear Schr\"odinger equation.
Dokl. Akad. Nauk SSSR \textbf{259} (1981), no. 1, 76--79. %MR0627387

\bibitem{JerrardSmets} R. L. Jerrard and D. Smets,
On the motion of a curve by its binormal curvature.
J. Eur. Math. Soc. (JEMS) \textbf{17} (2015), no. 6, 1487--1515. %MR3353807

\bibitem{KVZ} R. Killip, M. Visan, and X. Zhang,
Low regularity conservation laws for integrable PDE.
To appear in Geom. Funct. Anal.

%\bibitem{KirkLenzStaff} 
%K. Kirkpatrick, E. Lenzmann, and G. Staffilani,
%On the continuum limit for discrete NLS with long-range lattice interactions.
%Comm. Math. Phys. \textbf{317} (2013), no. 3, 563--591.

\bibitem{Kishimoto}
N. Kishimoto,
A remark on norm inflation for nonlinear Schr\"odinger equations.
Preprint \texttt{arXiv:1806.10066}.

\bibitem{KochTataru} H. Koch and D. Tataru,
Conserved energies for the cubic NLS in 1-d.
Preprint \texttt{arXiv:1607.02534}.

\bibitem{Laks} M. Lakshmanan, Continuum spin system as an exactly solvable dynamical system.
Phys. Lett. A \textbf{61} (1977), no. 1, 53--54.

\bibitem{LL} L. Landau and E. Lifshitz,
On the theory of the dispersion of magnetic permeability in ferromagnetic bodies.
Phys. Z. Sowjet. \textbf{8} (1935), 153--169.

\bibitem{LR} O. E. Lanford and D. Ruelle,
Observables at infinity and states with short range correlations in statistical mechanics.
Comm. Math. Phys. \textbf{13} (1969), no. 3, 194--215.

\bibitem{LRS} J. L. Lebowitz, H. A. Rose, and E. R. Speer,
Statistical mechanics of the nonlinear Schr\"odinger equation.
J. Stat. Phys. \textbf{50} (1988), no. 3--4, 657--687.

\bibitem{LL9}  E. M. Lifshitz and L. P. Pitaevskii, \emph{Statistical physics. Part 2. Theory of the condensed state.}
Course of theoretical physics Vol. 9.  Translated from the Russian by J. B. Sykes and M. J. Kearsley. Pergamon Press, Oxford-Elmsford, N.Y., 1980. %MR0586944

\bibitem{Lozano}  G. I. Lozano, \emph{Poisson geometry of the Ablowitz-Ladik equations.}
Ph.D. thesis, The University of Arizona, 2004. %MR2706394

\bibitem{Magri} F. Magri,
A simple model of the integrable Hamiltonian equation.
J. Math. Phys. \textbf{19} (1978), no. 5, 1156--1162.  %MR0488516

\bibitem{Koiso} N. Koiso,
The vortex filament equation and a semilinear Schr\"odinger equation in a Hermitian symmetric space.
Osaka J. Math. \textbf{34} (1997), no. 1, 199--214. %MR1439006

\bibitem{Peter} P. Petersen, Classical Differential Geometry.  Lecture notes, available from the authors web-page: \verb!http://www.math.ucla.edu/~petersen/DGnotes.pdf!

\bibitem{RobThom} J. A. G. Roberts and C. J. Thompson,
Dynamics of the classical Heisenberg spin chain.
J. Phys. A \textbf{21} (1988), no. 8, 1769--1780. %MR0952724

\bibitem{RodRub} I. Rodnianski, Y. A. Rubinstein, and G. Staffilani,
On the global well-posedness of the one-dimensional Schr\"odinger map flow.
Anal. PDE \textbf{2} (2009), no. 2, 187--209. %MR2547134

\bibitem{Skly} E. K. Sklyanin, Some algebraic structures connected with the Yang-Baxter equation. Functional Anal. Appl. \textbf{16} (1982), no. 4, 263--270. %MR0684124

\bibitem{Bardos&c}  P.-L. Sulem, C. Sulem, and C. Bardos,
On the continuous limit for a system of classical spins.
Comm. Math. Phys. \textbf{107} (1986), no. 3, 431--454. %MR0866199

\bibitem{Takh} L. A. Takhtajan,
Integration of the continuous Heisenberg spin chain through the inverse scattering method.
Phys. Lett. A \textbf{64} (1977), no. 2, 235--237. %MR0456051

\bibitem{Vaninsky} K. L. Vaninsky,
An additional Gibbs' state for the cubic Schr\"odinger equation on the circle.
Comm. Pure Appl. Math. \textbf{54} (2001), no. 5, 537--582. %MR1811126

\bibitem{ZakhTakh} V. E. Zakharov and L. A. Takhtadzhyan, Equivalence of the nonlinear Schr\"odinger equation and the equation of a Heisenberg ferromagnet.
Theor. Math. Phys. \textbf{38} (1979), no. 1, 17--23.
\end{thebibliography}
\end{document}